\documentclass[
final
]{dmtcs-episciences}

\usepackage[utf8]{inputenc}
\usepackage{subfigure}
\usepackage{here}

\newtheorem{theorem}{Theorem}[section]
\newtheorem{lemma}[theorem]{Lemma}

\newtheorem{conjecture}[theorem]{Conjecture}

\newtheorem{claim}[theorem]{Claim}
\newtheorem{problem}[theorem]{Problem}

\author[C.J. Casselgren et. al]{Carl Johan Casselgren\affiliationmark{1}\thanks{Research supported by a grant from the Swedish Research council VR
(2017-05077)}
  \and Jonas B. Granholm\affiliationmark{1}
  \and Fikre B. Petros\affiliationmark{2}}
\title[Extending partial edge colorings of iterated cartesian products]{Extending partial edge colorings of iterated
cartesian products of cycles and paths}
\affiliation{
  Link\"opings Universitet, Sweden\\
  Addis Ababa University, Ethiopia}
\begin{document}
\publicationdata{vol. 26:2}{2024}{5}{10.46298/dmtcs.11377}{2023-05-25; 2023-05-25; 2024-03-08}{2024-04-19}

\maketitle
\begin{abstract}We consider the problem of extending partial edge colorings of iterated cartesian
products of even cycles and paths, focusing on the case when the precolored
edges satisfy either an Evans-type condition or is a matching.
In particular, we prove that if $G=C^d_{2k}$ is the $d$th power of the
cartesian product of the even cycle $C_{2k}$ with itself, and at most
$2d-1$ edges of $G$ are precolored, then there is a proper $2d$-edge coloring
of $G$ that agrees with the partial coloring. We show that the same conclusion
holds, without restrictions on the number of precolored edges, if any two
precolored edges are at distance at least $4$ from each other.
For odd cycles of length at least $5$, we prove that if $G=C^d_{2k+1}$ 
is the $d$th power of the
cartesian product of the odd cycle $C_{2k+1}$ with itself ($k\geq2$), and at most
$2d$ edges of $G$ are precolored, then there is a proper $(2d+1)$-edge coloring
of $G$ that agrees with the partial coloring.
Our results generalize previous ones on 
precoloring extension of hypercubes
[Journal of Graph Theory 95 (2020) 410--444].
\end{abstract}

\section{Introduction}

	An {\em (edge) precoloring} (or {\em partial edge coloring})
	of a graph $G$ is a proper edge coloring of some
	subset $E' \subseteq E(G)$; {\em a $t$-edge precoloring}
	is such a coloring with $t$ colors.
	A  $t$-precoloring $\varphi$ of $G$ is
	{\em extendable} if there is a proper $t$-edge coloring $f$ of $G$
	such that $f(e) = \varphi(e)$ for any edge $e$ that is colored
	under $\varphi$; $f$ is called an {\em extension}
	of $\varphi$.
	In general, the problem of extending a given edge precoloring
	is an $\mathcal{NP}$-complete problem,
	already for $3$-regular bipartite graphs \cite{EastonParker, Fiala}.
	
	Edge precoloring extension problems seem to have been 
	first considered in connection with the problem of completing partial
	Latin squares and the well-known Evans' conjecture 
	that every $n \times n$ partial Latin square with at most $n-1$ non-empty
	cells is completable to a Latin square \cite{Evans}. By a well-known 
	correspondence, the problem of completing
	a partial Latin square is equivalent to asking if a partial 
	edge coloring with 
	$\Delta(G)$ colors of a balanced complete bipartite graph $G$ is extendable
	to a $\Delta(G)$-edge coloring, where $\Delta(G)$ 
	as usual denotes the maximum degree.
	 Evans' conjecture was proved for large $n$ by H\"aggkvist \cite{Haggkvist78},
	and in full generality by Andersen and Hilton \cite{AndersenHilton}, and, 
	independently, by
	Smetaniuk \cite{Smetaniuk}.
	
	Another early reference on
	edge precoloring extension is \cite{MarcotteSeymour}, where the authors 
	study the problem from the viewpoint of polyhedral combinatorics.
	More recently, the problem of extending a precoloring of
	a matching has been considered
	in \cite{EGHKPS}.
	In particular, it is conjectured that for every graph $G$,
	if $\varphi$ is an edge precoloring of a matching $M$ in $G$
	using $\Delta(G)+1$ colors,
	and any two edges in $M$ 
	are at distance at least $2$ from each other,  then $\varphi$ 
	can be extended to a proper $(\Delta(G)+1)$-edge coloring of $G$;
	here,
	by the {\em distance} between two
	edges $e$ and $e'$ we mean the number of edges in a 
	shortest path between an endpoint of $e$ and an endpoint of $e'$; 
	a {\em distance-$t$ matching} is a matching where any
	two edges are at distance at least $t$ from each other.
	In \cite{EGHKPS}, it is proved that this conjecture holds
	for e.g. bipartite multigraphs and subcubic multigraphs, and
	in \cite{GiraoKang} it is proved that a version of the conjecture with the
	distance increased to 9 holds for general graphs.

	Quite recently, with motivation
	from results on completing partial Latin squares,
	questions on extending partial edge colorings
	of $d$-dimensional hypercubes $Q_d$ were studied 
	in \cite{CasselgrenMarkstromPham}. 
	Among other things, a characterization of
	partial edge colorings with at most $d$ precolored edges that are extendable
	to $d$-edge colorings of $Q_d$ is obtained, thereby establishing an analogue
	for hypercubes of the characterization by
	Andersen and Hilton \cite{AndersenHilton} of 
	$n \times n$ partial Latin 
	squares with at most $n$ non-empty cells that are completable to Latin squares.
	In particular, every partial $d$-edge coloring with at most $d-1$ colored edges
	is extendable to a $d$-edge coloring of $Q_d$.
	This line of investigation was continued in 
	\cite{CasselgrenPetros, CasselgrenPetros2} where similar questions are
	investigated for trees.
	
	In \cite{CasselgrenPetrosFufa}, similar questions are investigated for
	{\em cartesian products} of graphs.
	The {\em cartesian product} $G \square H$ of two graphs $G$ and $H$ is
	the graph with vertex set $V(G \square H) = \{(u,v) : u \in V(G), v \in V(H)\}$,
	and where $(u,v)$ is adjacent to $(u',v')$ if and only if $u=u'$ and $vv' \in E(H)$,
	or $uu' \in E(G)$ and $v=v'$.

	In \cite{CasselgrenPetrosFufa}, Evans-type edge precoloring extension results
	are obtained for the cartesian products of complete and complete bipartite graphs
	with $K_2$, respectively, 
	as well as for the product of $K_2$ with graphs of small maximum
	degree and trees. Moreover, similar results for the cartesian product of
	$K_2$ with a general regular (triangle-free) graph, where the precolored edges are
	required to be independent, were obtained.
	
	In this paper, we continue the study of questions on precoloring extension
	of cartesian products of graphs with a focus on 
	iterated cartesian products of graphs.
	Denote by $G^d$ the $d$th power of the cartesian product of $G$ with itself.
	We pose the following question.
	
	\begin{problem}
	\label{prob:general}
		Let $G$ be a graph where every precoloring of at 
		most $\chi'(G)-k$ edges, where $k \geq 1$,
		can be extended to a proper $\chi'(G)$-edge coloring. Is it true that every
		precoloring of at most $\chi'(G^d)-k$ edges of $G^d$ can be extended to a
		$\chi'(G^d)$-edge coloring of $G^d$?
	\end{problem}
		The result of \cite{CasselgrenMarkstromPham} for hypercubes 
		deals with the case when $G=K_2$ (as well as $G=C_4$), so
		a positive answer to Problem \ref{prob:general} would be a far-reaching
		generalization of this result.
	
In this paper, we study Problem \ref{prob:general} for graphs with maximum
degree two. We verify that it has a positive answer for even as well as for
odd cycles of length at least $5$, 
and therefore also for paths. 
The case of odd cycles of length $3$ appears to be more difficult, and it remains
an open problem whether Problem \ref{prob:general} has a positive answer in this case.

Even though any partial edge coloring of an odd cycle is extendable,
we shall restrict ourselves to the case when at most $\chi'(G)-1$ edges in a graph
$G$ are precolored, since for all connected graphs except odd cycles and stars, there
are examples of
partial edge colorings with $\chi'(G)$ precolored edges that are not extendable.
In fact, in \cite{CasselgrenPetrosFufa} it was proved that every partial 
$\chi'(G)$-edge coloring of $G$ is extendable if and only $G$ is isomorphic to a star $K_{1,n}$ 
or an odd cycle.

For even cycles, we additionally prove that any precoloring of a distance-$4$
matching in $C^d_{2k}$ is extendable to a proper $2k$-edge coloring.
Here the argument relies heavily on the fact that $C^d_{2k}$ is Class 1, and
we do not know whether a similar result hold for odd cycles.


\section{Preliminaries}

Before we prove our results, let us introduce some terminology and auxiliary
results.

If $\varphi$ is an edge precoloring of $G$ 
and an edge $e$ is colored under $\varphi$,
then we say that $e$ is {\em $\varphi$-colored}.
A color $c$ {\em appears} at a vertex $v$ under $\varphi$ if
there is an edge incident with $v$ that is colored $c$; otherwise, $c$ is
{\em missing} at $v$.

If the edge coloring $\varphi$ uses $t$ colors and
$1 \leq a,b \leq t$, then a path or cycle in
$G$ is called {\em $(a,b)$-colored under $\varphi$} if
its edges are colored by colors $a$ and $b$ alternately.
We also say that such a path or cycle is \emph{bicolored under $\varphi$}.
By switching colors
$a$ and $b$ on a maximal $(a,b)$-colored path or an $(a,b)$-colored cycle,
we obtain another proper $t$-edge coloring of $G$;
this operation is called an {\em interchange} or a {\em swap}.
	
In the above definitions, we often leave out the 
reference to an explicit coloring $\varphi$, if the coloring
is clear from the context.

If $G_1$ and $G_2$ are subgraphs of $G$, and $f_i$ is a proper edge coloring
of $G_i$, then we say that $f_1$ {\em has no conflicts} with $f_2$ if
no vertex is incident with two edges $e_1$ and $e_2$ such that $f_1(e_1) = f_2(e_2)$.

By construction, $G=C^d_{r}$ decomposes into $d$ subgraphs in terms of its edges, 
each consisting of
$r^{d-1}$ disjoint copies of $C_r$; these subgraphs are called {\em dimensions}.
Each subgraph of a dimension which is isomorphic to $C_r$ is called a {\em layer},
and each component of $G-E(D)$, where $D$ is a dimension, is called a {\em plane} 
of $G$. If $d=2$, then layers and planes are identical objects.

\begin{figure} [H]
\centering
\includegraphics[scale=1.0]{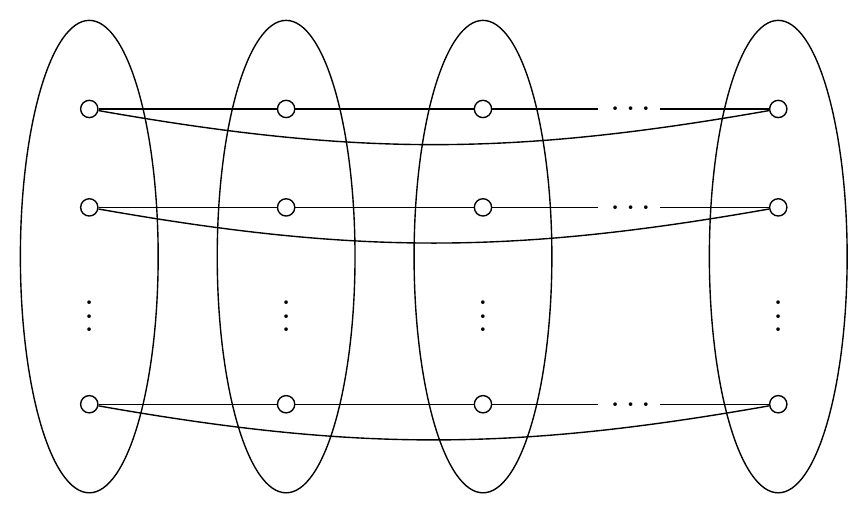}
\caption{An illustration of dimensions, layers and planes. Each cycle forms a {\em layer},
all the cycles together form a {\em dimension}, and the components obtained by removing all the edges from the cycles are the {\em planes}.}
\label{fig22}
\end{figure}

In Figure \ref{fig22}, the edge-induced subgraph consisting of all vertices and
drawn edges form a dimension, each cycle is a layer, and each connected component
in the subgraph obtained by removing all drawn edges is a plane.

Two planes are {\em adjacent} if there is an edge with endpoints
in both planes. Similarly an edge $e$ not contained in a plane is 
{\em incident} to the plane if one endpoint of $e$ is contained in the plane,
and we say that a layer edge is {\em between} two planes if it is incident with both
planes.

Two vertices of two distinct planes are {\em corresponding} if they are joined
by an edge; similarly for edges. Given edge colorings of two distinct planes,
we say that the planes are colored {\em correspondingly} if corresponding edges
have the same color.

We shall also need some standard definitions on list edge coloring.
	Given a graph $G$, assign to each edge $e$ of $G$ a set
	$\mathcal{L}(e)$ of colors.
	Such an assignment $\mathcal{L}$ is called
	a \emph{list assignment} for $G$ and
	the sets $\mathcal{L}(e)$ are referred
	to as \emph{lists} or \emph{color lists}.
	If all lists have equal size $k$, then $\mathcal{L}$
	is called a \emph{$k$-list assignment}.
	Usually, we seek a proper
	edge coloring $\varphi$ of $G$,
	such that $\varphi(e) \in \mathcal{L}(e)$ for all
	$e \in E(G)$. If such a coloring $\varphi$ exists then
	$G$ is \emph{$\mathcal{L}$-colorable} and $\varphi$
	is called an \emph{$\mathcal{L}$-coloring}.
	Denote by $\chi'_L(G)$ the minimum integer $t$
	such that $G$ is $\mathcal{L}$-colorable
	whenever $\mathcal{L}$ is a $t$-list assignment.
	If $\chi'_L(G) \leq t$, then $G$ is {\em $t$-edge-choosable}.
	The following lemmas are well-known and easy to prove.

\begin{lemma}
	\label{lem:evencycle}
	Every even cycle is $2$-edge-choosable.
\end{lemma}

\begin{lemma}
\label{lem:oddcycle}
	If $L$ is a $2$-list assignment for the edges of an odd cycle $C$, then
	$C$ is $L$-colorable, unless all lists are identical.
\end{lemma}

We shall also use the well-known proposition that paths are edge-list colorable from 
a list assignment where every edge except the first one gets a list of size
at least two.


\section{Extension of $2d-1$ precolored edges of $C^d_{2k}$}

In this section, we prove the following theorem.

\begin{theorem}
\label{th:evencycle}
	If $G=C^d_{2k}$ is the $d$th power of the cartesian product of the even cycle
	$C_{2k}$ with itself, and $\varphi$ is a proper partial edge coloring of 
	$G$ with at most $2d-1$ precolored edges, then $\varphi$ can be extended
	to a proper $2d$-edge coloring of $G$.
\end{theorem}

As mentioned in the introduction, every connected graph except odd cycles and stars
have a partial edge coloring with $\chi'(G)$ precolored edges that is not extendable.
Thus, since $\chi'(G) =2d$, the bound on the number of precolored edges here is best possible.

\begin{proof}[Proof of Theorem \ref{th:evencycle}]
	The proof  proceeds by induction on $d$, the case $d=1$
	being trivial.
	We shall prove a series of lemmas that together will imply the theorem.
	In the proofs of these lemmas we shall consider a specified dimension $D_1$, and
	 the subgraph $G-E(D_1)$ consisting of $2k$ planes
	$Q_1, \dots, Q_{2k}$, where $Q_i$ is adjacent to $Q_{i+1}$ 
	(here, and in the following,
	indices are taken modulo $2k$). 

	We shall assume that every precoloring of a plane of
	$G-E(D_1)$ with at most $2d-3$ precolored
	edges is extendable to a proper edge coloring using $2d-2$ colors, and  prove that
	a given precoloring $\varphi$ of $G$ with at most $2d-1$ precolored edges is extendable to
	a proper $2d$-edge coloring of $G$.
		
		We shall distinguish between the following cases, each of which
		is dealt with in a lemma below.

	\begin{itemize}
		
		\item There is a dimension of $G$ that contains no precolored edges.
		
		\item Each dimension of $G$ contains precolored edges, 
	and there is a dimension
		with at most two precolored edges, the colors of which do not appear on
		edges in any other dimension of $G$.
		
		\item Every dimension of $G$
		contains edges with colors that also appear on edges in another dimension,
		or at least three precolored edges.
		\end{itemize}
	\end{proof}
	\pagebreak
	\begin{lemma}
	\label{cl:oneempty}
		If there is a dimension of $G$ that contains no precolored edges, then $\varphi$
		is extendable.
	\end{lemma}
	\begin{proof}
	Suppose that $D_1$ is a dimension in $G$ that contains no precolored edges,
	and consider the subgraph $G-E(D_1)$.
		
	Suppose first that all precolored edges are contained in one plane, say $Q_1$.
	Let $c_1$ and $c_2$ be two colors used by $\varphi$ (if just one color appears
	under $\varphi$, then $c_2$ is any color from $\{1,\dots,2d\} \setminus \{c_1\}$).
	From the restriction of $\varphi$ to $Q_1$, we define an edge 
	precoloring $\varphi'$ of $Q_1$ by removing the colors $c_1$ and $c_2$ 
	from any edge of $Q_1$ $\varphi$-colored by these colors.
	Then, by the induction hypothesis, $\varphi'$
	is extendable
	to a $(2d-2)$-coloring of $Q_1$ using colors 
	$\{1,\dots, 2d\} \setminus \{c_1,c_2\}$. Next we recolor the edges 
	$\varphi$-precolored $c_1$ and $c_2$ by these colors, and thereafter
	color all other planes correspondingly.
	Thus, we can 
	define a list assignment $L$ for the edges of $D_1$,
	by for each edge $e\in E(D_1)$, letting $L(e)$ be the set of all colors
	from $\{1,\dots,2d\}$ that do not appear on edges that are adjacent to $e$.
	By Lemma \ref{lem:evencycle}, we
	can properly color the edges of $D_1$ from these lists to obtain a proper coloring
	that has no conflicts with the coloring of $G-E(D_1)$, 
	and thus  $\varphi$ is extendable.
		
		\bigskip
		
	Next, we consider the case when exactly two planes, say $Q_1$ and $Q_i$
	contain all precolored edges. Since at most $2d-1$ colors appear under
	$\varphi$, there is a color $c_1 \in \{1,\dots, 2d\}$ that is not used by $\varphi$.
	Furthermore, let $c_2$ be a color appearing on some edge in the plane
	with the largest number of precolored edges, say $Q_1$.
	Let $\varphi'$ be the coloring obtained from $\varphi$ by removing color $c_2$
	from any edge colored $c_2$ under $\varphi$.
	Then the restrictions of $\varphi'$ to $Q_1$ and $Q_i$, respectively,
	are extendable to proper $(2d-2)$-edge colorings using colors
	$\{1,\dots, 2d\} \setminus \{c_1,c_2\}$.
	By recoloring any edge $\varphi$-colored $c_2$
	by the color $c_2$, we obtain proper edge colorings $f_1$ and $f_i$
	of $Q_1$ and $Q_i$, respectively.
		
	Now, either $i\neq 2$ or $i \neq 2k$; suppose the former holds.
	Then we color $Q_2$ correspondingly to how $Q_1$ is colored under $f_1$,
	and we color all other uncolored $Q_j$'s correspondingly to how $Q_i$ is colored
	under $f_i$.
	Now, since $Q_{2j-1}$ and $Q_{2j}$ are colored correspondingly,
	for every edge $e$ with one endpoint in $Q_{2j-1}$ and one endpoint in $Q_{2j}$,
	there is a color $\{1,\dots, 2d\} \setminus \{c_1\}$ that does not appear
	at an endpoint of $e$. Thus, by coloring all such edges by such a color
	and then coloring all other edges of $D_1$ by the color $c_1$, 
	we obtain an extension of $\varphi$.
		
		\bigskip
		
		Lastly, let us consider the case when at least three planes contain
		all precolored edges. As before, let $c_1$ be a color that is not 
		used by $\varphi$.
		Since at least three planes contain precolored edges, 
		each plane contains at most $2d-3$ precolored edges, and
		two adjacent planes
		contain precolored edges from at most $2d-2$ colors. 
		This implies that for each $j=1,\dots,k$  there is a color $c'_j$ in
		$\{1,\dots, 2d\} \setminus \{c_1\}$ that is not used in the restriction of $\varphi$ to
		$Q_{2j-1}$ and $Q_{2j}$.
		Thus for $j=1, \dots, k$, we can extend the restriction of
		$\varphi$ to $Q_{2j-1}$ and $Q_{2j}$ using the $2d-2$ colors in
		$\{1,\dots, 2d\} \setminus \{c_1,c'_j\}$. 
		For $j=1,\dots k$, we then color all edges of $D_1$
		between $Q_{2j-1}$ and $Q_{2j}$ by $c'_j$, and all other edges of $D_1$ by
		the color $c_1$. This yields an extension of $\varphi$.
	\end{proof}
	
	\begin{lemma}
	\label{cl:noempty2free}
		If each dimension of $G$ contains precolored edges and there is a dimension
		with at most two precolored edges, the colors of which do not appear on
		edges in any other dimension of $G$,
		then $\varphi$ is extendable.
	\end{lemma}
	\begin{proof}
		Assume first that there is a dimension $D_1$ containing only one precolored
		edge and that there is a plane $Q_1$ in $G-E(D_1)$ containing all other 
		precolored edges. Suppose that the precolored edge of $D_1$ is colored $c_1$.
		As in the proof of the preceding lemma, 
		there is an extension of the restriction of $\varphi$ to $Q_1$
		using colors $\{1,\dots, 2d\} \setminus \{c_1\}$.
		Next, we color all other planes of $G-E(D_1)$ correspondingly,
		which implies that every edge of $D_1$ is adjacent to edges of $2d-2$
		different colors, so by Lemma \ref{lem:evencycle},
		$\varphi$ is extendable.
		
	Let us now consider the case when there is no such dimension containing only one
	precolored edge and a plane containing all other precolored edges.
	Let $D_1$ be a dimension containing at most two edges that are precolored
	by colors not appearing on any other edges under $\varphi$.
	Our assumption implies that every plane in $G-E(D_1)$ contains at most $2d-3$ 
		precolored edges, and that there are two colors $c_1, c_2$ that do not appear
		on any edge in $G-E(D_1)$, and at most two edges of $D_1$ are colored
		by colors from $\{c_1,c_2\}$.
		
		Suppose first that there is an extension of the restriction of $\varphi$
		to $D_1$ using colors $c_1$ and $c_2$.
		Since every plane in $G-E(D_1)$ contains at most $2d-3$ precolored edges,
		every plane has a proper edge coloring using colors 
		$\{1,\dots, 2d\} \setminus \{c_1,c_2\}$ that agrees with $\varphi$.
		Hence, $\varphi$ is extendable.
		
		Suppose now that the restriction of $\varphi$ to $D_1$ is not extendable
		using colors $c_1$ and $c_2$. Then there are at least two precolored edges in
		$D_1$, and so $G-E(D_1)$ contains at most $2d-3$ precolored edges and there
		is a color $c_3 \in \{1,\dots, 2d\} \setminus \{c_1,c_2\}$ that does not appear on any edge under
		$\varphi$.
		
		Since the restriction of $\varphi$ to $D_1$ is not extendable, there are
		planes $Q_1, Q_{2}, \dots, Q_i$ in $G-E(D_1)$ such that
		$Q_1$ and $Q_i$ are incident with precolored edges of $D_1$ and
		$Q_2,\dots, Q_{i-1}$ are not.
		Without loss of generality we assume that $Q_1$ is incident with an edge of
		$D_1$ that is precolored $c_1$.
		We take an extension of $Q_1$ using colors 
		$\{1, \dots, 2d\} \setminus \{c_1,c_3\}$,
		and for $j=2,\dots, i-1$, we take an extension of the restriction of $\varphi$
		to $Q_j$ using colors $\{1, \dots, 2d\} \setminus \{c_2,c_3\}$. 
		Moreover, we color
		every path in $D_1$ with vertices in $Q_1,\dots, Q_i$ using colors
		$c_3$ and $c_2$ alternately, and starting with $c_3$ at $Q_1$.
		
		If $Q_i$ is incident with an edge of $D_1$ colored $c_2$, then
		$i$ is even, and
		we take
		extensions of the restriction of $\varphi$ to $Q_i$ and $Q_{i+1}$ using colors
		$\{1, \dots, 2d\} \setminus \{c_2,c_3\}$, and
		for
		$j=i+2, \dots, 2d$,
		we take extensions of the restrictions of $\varphi$
		to $Q_j$ using colors
		$\{1, \dots, 2d\} \setminus \{c_1,c_3\}$.
		Moreover, we color every path in $D_1$ with vertices in 
		$Q_{i+1},\dots, Q_{2k}$ using colors
		$c_3$ and $c_1$ alternately, and starting with $c_3$ at $Q_{i+1}$.
		Finally, we color all edges between $Q_1$ and $Q_{2k}$ by the color $c_1$,
		and all edges between $Q_i$ and $Q_{i+1}$ by the color $c_2$.
		This yields an extension of $\varphi$.
		
		If $Q_i$ is incident with an edge colored $c_1$, then $i$ is odd,
		and we proceed
		similarly, but take extensions of the restrictions of $\varphi$ to
		$Q_i$ and 
		$Q_{i+1}$ using colors $\{1, \dots, 2d\} \setminus \{c_1,c_2\}$,
		for $j=i+2,\dots 2k-1$, we 
		take extensions of the restrictions of $\varphi$
		to $Q_j$ using colors
		$\{1, \dots, 2d\} \setminus \{c_2,c_3\}$, and an extension of the 
		restriction of $\varphi$
		to $Q_{2k}$ using colors $\{1,\dots, 2d\} \setminus \{c_1,c_3\}$.
		We then color the paths of $D_1$ so that the resulting coloring
		is proper and agrees with $\varphi$.
	\end{proof}
	
	\begin{lemma}
	\label{cl:noemptynofree}
		If each dimension of $G$
		contains edges with colors that also appear on edges in another dimension,
		or at least three precolored edges,
		then $\varphi$ is extendable.
	\end{lemma}
	\begin{proof}
		Since at most $2d-1$ edges are precolored, there is a dimension $D_1$
		with just one precolored edge $e$. Suppose $\varphi(e)=c_1$.
		Since at least one color appears on at least two edges, there are two
		colors $c_2, c_3 \in \{1,\dots, 2d\} \setminus \{c_1\}$ that 
		do not appear on any edge under $\varphi$.
		
		We consider some different cases.
		\bigskip

		{\bf Case 1.} {\em All precolored edges except $e$ lie in the same plane}:
		
		Let $Q_1$ be a plane containing all precolored edges except $e$. 
		By removing the color from every edge of $Q_1$ that is $\varphi$-colored $c_1$,
		we obtain a precoloring that by the induction hypothesis is extendable
		to a proper edge coloring of $Q_1$ using colors 
		$\{1,\dots, 2d\} \setminus \{c_1,c_2\}$.
		Denote this coloring by $f$.
		By recoloring the edges of $Q_1$ that are $\varphi$-colored $c_1$, we obtain, from $f$,
		an extension $f'$ of the restriction of $\varphi$ to $Q_1$.
		We color all other planes of $G-E(D_1)$ correspondingly.
		
		Suppose first that $e$ is incident with $Q_1$.
		Then, by Lemma \ref{lem:evencycle}, there is proper edge coloring $g$ of
		$D_1$ that has no conflicts with $f'$.
		Moreover, since just one edge of $D_1$ is precolored, we can choose this
		coloring so that $g(e) = c_1$. Hence, $\varphi$ is extendable.
		
		Next, we consider the case when $e$ is not incident with $Q_1$.
		Let $C$ be the cycle in $D_1$ containing $e$. If no vertex of $C$
		is incident with an edge colored $c_1$ under $f'$, then we may proceed
		as in the preceding paragraph. Otherwise, the endpoints of $e$ are incident with
		two edges $e_1$ and $e_2$ of $G-E(D_1)$ colored $c_1$.
		Note that $e, e_1, e_2$ are contained in a $4$-cycle in $G$, the fourth edge of which
		we denote by $e_3$.
		As before, we properly color the edges of $D_1$ so that the resulting coloring $g'$
		of $G$ is proper.
		Moreover, we choose $g'$
		so that $g'(e)= g'(e_3)= c_2$. By swapping colors on the bicolored $4$-cycle
		with edges $e, e_1,e_2,e_3$ we finally obtain an extension of $\varphi$.
		
		\bigskip
				
		{\bf Case 2.} {\em All precolored edges except $e$ lie in exactly two planes}:
		
		Let $Q_1$ and $Q_i$ be the two planes containing the precolored edges
		distinct from $e$. 
		
		Suppose first that $e$ is not adjacent to $Q_1$ or $Q_i$. 
		Since each of $Q_1$ and $Q_i$ contains at most $2d-3$ precolored edges,
		there are extensions
		$f_1$ and $f_i$ of the restrictions of $\varphi$ to $Q_1$ and $Q_i$, respectively,
		using colors $\{1,\dots, 2d\} \setminus \{c_2,c_3\}$.
		Next, we properly color 
		each plane $Q$ of $G-E(D_1)$ that is distinct from $Q_1$ and $Q_i$,
		by colors $\{1,\dots, 2d\} \setminus \{c_2,c_3\}$, so that all these planes
		are colored correspondingly.
		Moreover, color all edges of $D_1$ alternately using colors $c_2$ and $c_3$
		and starting with color $c_2$ for all edges of $D_1$ with endpoints in $Q_1$
		and $Q_2$. Then $e$ is contained in 
		a bicolored $4$-cycle,
		and by swapping colors on this cycle,
		we obtain an extension of $\varphi$.
		
		\bigskip
		
		Let us now consider the case when $e$ is adjacent to exactly one of
		$Q_1$ and $Q_i$. Suppose
		e.g. that $e$ is incident with $Q_1$ and $Q_2$, so $i \neq 2$.
		Since each of $Q_1$ and $Q_i$ contains at most $2d-3$ precolored edges,
		there are extensions
	$f_1$ and $f_i$ of the restrictions of $\varphi$ to $Q_1$ and $Q_i$, respectively,
	using colors $\{1,\dots, 2d\} \setminus \{c_2,c_3\}$.
	Next, we color all uncolored planes in $G-E(D_1)$ 
	correspondingly to how $Q_1$ is colored, color all cycles of
	$D_1$ properly using colors $c_2$ and $c_3$, and starting with color $c_2$
	for all edges with endpoints in $Q_1$ and $Q_2$. Since $Q_1$
	and $Q_2$ are colored correspondingly, there is a bicolored $4$-cycle with
	colors $c_1, c_2$ containing $e$. By swapping colors on this $4$-cycle,
	we obtain an extension of $\varphi$.

		\bigskip
		
	Suppose now that $e_1$ is incident with both planes containing precolored edges,
	say $Q_1$ and $Q_2$. By removing the color from any edge that is colored $c_1$
	under $\varphi$, we obtain a precoloring $\varphi'$ of $G$ such that the restrictions
	of $\varphi'$ to $Q_1$ and $Q_2$, respectively, are 
	extendable to proper edge colorings
	using colors $\{1,\dots, 2d\} \setminus \{c_1,c_3\}$. By recoloring any
	edge of $Q_1$ and $Q_2$ that is $\varphi$-colored $c_1$, 
	we obtain proper edge colorings
	$f_1$ and $f_2$ of $Q_1$ and $Q_2$, respectively, using colors 
	$\{1,\dots, 2d\} \setminus \{c_3\}$. We color the edges of $D_1$ between
	$Q_1$ and $Q_2$ by $c_3$, except for $e$ which is colored $c_1$.

		Next,
		we color $Q_{3} , \dots, Q_{2k-1}$ correspondingly to 
		how $Q_2$ is colored under $f_2$ 
		and $Q_{2k}$ correspondingly to how $Q_1$ is colored. Now, since $Q_2$ and
		$Q_3$ are colored correspondingly, for each edge $e'$ of $D_1$ between
		$Q_2$ and $Q_3$, there is a color $c' \in \{1,\dots, 2d\}$ 
		that does not appear on an adjacent edge
		in $Q_2$ or on the edge between $Q_1$ and $Q_2$. We color every such edge $e'$
		between $Q_2$ and $Q_3$
		by such a color $c'$, and thereafter color all edges of every path
		of $D_1$ from $Q_1$ to $Q_{2k}$ alternately by the colors used 
		on the edges between $Q_1$
		and $Q_2$, and $Q_2$ and $Q_3$ respectively.
		This yields a proper partial edge coloring where the endpoints of every edge 
		of $D_1$ between $Q_1$ and $Q_{2k}$ are incident to edges with $2d-1$ colors
		from $\{1,\dots, 2d\}$. Thus we can properly color every such edge so that 
		we get an extension of $\varphi$ using colors $1,\dots, 2d$.

		\bigskip

		{\bf Case 3.} {\em At least three planes in $G-E(D_1)$ contain
		precolored edges}:
		
		The assumption implies that every plane in $G-E(D_1)$ contains at most
		$2d-4$ precolored edges, and any two adjacent planes contain altogether at
		most  $2d-3$ precolored edges.
		Without loss of generality, we assume that $e$ is incident with $Q_1$
		and $Q_2$. Since every vertex of $Q_1$ and $Q_2$ has degree $2d-2$,
		and $Q_1 \cup Q_2$ contains at most $2d-3$ precolored edges,
		there are uncolored corresponding edges $e_1 \in E(Q_1)$ and $e_2 \in E(Q_2)$
		that are adjacent to $e$ but not to any edge in $Q_1 \cup Q_2$
		precolored $c_1$. From the restriction of $\varphi$ to $G-E(D_1)$,
		we define a new precoloring
		$\varphi'$ of $G-E(D_1)$ by in addition coloring $e_1$ and $e_2$ by the color 
		$c_1$.
		
		Now, since each component of $G-E(D_1)$ contains at most $2d-3$ 
		$\varphi'$-precolored edges, by the induction hypothesis, 
		there is an extension of $\varphi'$ to $G-E(D_1)$ using colors
		$\{1,\dots, 2d\} \setminus \{c_2, c_3\}$. Next, we properly color
		the edges of $D_1$ using colors $c_2$ and $c_3$, and starting with color
		$c_2$ for the edges with endpoints in $Q_1$ and $Q_2$.
		Now, since $e_1$ and $e_2$ are both colored $c_1$, there is a bicolored
		$4$-cycle with edges $e_1,e_2$ and $e$ with colors $c_1$ and $c_2$. By
		swapping colors on this $4$-cycle, we obtain an extension of $\varphi$.	
	\end{proof}


\section{Extending a precoloring of $2d$ edges in $C^d_{2k+1}$}

In this section, we prove the following theorem for the iterated cartesian 
product of odd cycles of length at least $5$.

\begin{theorem}
\label{th:oddcycle}
	If $G=C^d_{2k+1}$ is the $d$th power of the cartesian product of the odd cycle
	$C_{2k+1}$ with itself ($k\geq 2$), and $\varphi$ is a proper partial edge coloring of 
	$G$ with at most $2d$ precolored edges, then $\varphi$ can be extended
	to a proper $(2d+1)$-edge coloring of $G$.
\end{theorem}

As for the case of even cycles, (for $d \geq 2$) it is easily seen that the number of
precolored edges here is best possible, because $\chi'(G)=2d+1$ .



\begin{proof}[Proof of Theorem \ref{th:oddcycle}]
	The proof of this theorem is similar to the proof of 
	Theorem \ref{th:evencycle}, so 
	we shall omit or just sketch some parts which are similar to techniques in that proof.
	Particularly in the last parts of the proof, to avoid tedious repetition we omit parts which are
	very similar to techniques that have been described in more detail earlier in the proof.

We proceed by induction on $d$,  the case $d=1$ being trivial.
		As in the proof of Theorem \ref{th:evencycle},
	we shall prove a series of lemmas that together will imply the theorem.
	Since odd cycles are not $2$-edge-colorable, the proof is longer and more
	difficult than the proof of that theorem.
	In the proofs of these lemmas 
	we shall consider a specified dimension $D_1$, and
	the subgraph $G-E(D_1)$ consisting of $2k+1$ planes
		$Q_1, \dots, Q_{2k+1}$, where $Q_i$ is adjacent to $Q_{i+1}$
		(here, and in the following,
		indices are taken modulo $2k+1$).

	We shall assume that every edge precoloring of a plane of $G-E(D_1)$
	with at most $2d-2$ precolored
	edges is extendable to a proper edge coloring using $2d-1$ colors, and prove that
	a given precoloring $\varphi$ of $G$ with at most $2d$ precolored edges is extendable to
	a proper $(2d+1)$-edge coloring of $G$.
	To that end, we shall distinguish between the following different cases.

	\begin{itemize}
		
		\item There is a dimension of $G$ that contains no precolored edges.
		
		\item Every dimension of $G$ contains precolored edges, and there is a
		dimension with at most two precolored edges, the colors of which
		do not appear on edges in any other dimension of $G$.
		
		\item Every dimension of $G$ contains edges with colors that also appear
		on edges in other dimensions, or at least three precolored edges,
		and one dimension contains only one precolored edge.
		
		\item Every dimension of $G$ contains two precolored edges,
		at least one of which has a color appearing on edges in another 
		dimension.
		
	\end{itemize}

		\end{proof}

	\begin{lemma}
	\label{cl:ODDoneempty}
		If there is a dimension of $G$ that contains no precolored edges, then $\varphi$
		is extendable.
	\end{lemma}
	\begin{proof}
		Suppose that $D_1$ is a dimension in $G$ that contains no precolored edges. 
		We consider some different cases.
		
		\bigskip
		
		{\bf Case 1.} {\em All precolored edges are contained in one plane:}
		
		Suppose that all precolored edges are contained in one plane, say $Q_1$.
		Let $c_1$ and $c_2$ be two colors used by $\varphi$ (if just one color appears
		under $\varphi$, then $c_2$ is any color from 
		$\{1,\dots,2d+1\} \setminus \{c_1\}$).
		By removing the colors $c_1$ and $c_2$ from any edge colored by these colors,
		we obtain an edge precoloring $\varphi'$ of $Q_1$ that is extendable
		to a $(2d-1)$-coloring of $Q_1$ using colors
		$\{1,\dots, 2d+1\} \setminus \{c_1,c_2\}$. 
		Next we recolor the edge precolored $c_1$
		and $c_2$, respectively, using these colors, and thereafter
		color all other planes correspondingly.
		Since all planes are colored correspondingly, we can 
		apply Lemma \ref{lem:oddcycle} to
		properly color the edges of each layer of $D_1$ to obtain
		an extension of $\varphi$.
		
		\bigskip
		
		{\bf Case 2.} {\em All precolored edges are contained in two planes:}
		
		Suppose that $Q_1$ and $Q_i$
		contain all precolored edges. We shall consider three different cases.
		
		Suppose first that $2d-1$ precolored edges are contained in 
		the same plane, say $Q_1$,
		and that one edge $e_i$ in $Q_i$ is colored $c_1$. Let $c_2$ be 
		a color appearing on some edge
		in $Q_1$. 
		From the restriction of $\varphi$ to $Q_1$ we define a 
		precoloring $\varphi'$ of $Q_1$
		by removing color $c_2$ from every edge $\varphi$-colored $c_2$.
		Then $\varphi'$ is extendable to a proper $(2d-1)$-edge coloring using
		$2d-1$ colors from $\{1,\dots 2d+1\} \setminus \{c_2\}$. By recoloring every edge
		of $Q_1$ that is $\varphi$-colored $c_2$ by the color $c_2$, 
		we obtain a proper edge coloring
		$f$ of $Q_1$.
		
		Let $e_1$ be the edge of $Q_1$ corresponding to $e_i$ of $Q_i$. If $f(e_1) = c_1$,
		then we color all planes in $G-E(D_1)$ correpondingly to how $Q_1$ is colored. By
		Lemma \ref{lem:oddcycle}, we may then color the edges of $D_1$ 
		to obtain an extension
		of $\varphi$. If, on the other hand, $f(e_1) = c_3 \neq c_1$, then we define
		a proper edge coloring of $Q_i$ by coloring it correspondingly to $Q_1$
		but permuting the colors in $\{1,\dots, 2d+1\}$
		so that $c_1$ is mapped to $c_3$ and vice versa, and all other
		colors are mapped to themselves.
		This yields a proper edge coloring of $Q_i$ that agrees
		with the restriction of $\varphi$ to $Q_i$. 
		We color all other $Q_j$'s correspondingly to
		how $Q_i$ is colored, and
		applying Lemma \ref{lem:oddcycle}, we obtain an extension of $\varphi$, as before.
		
	Let us now assume that both $Q_1$ and $Q_i$ contain at most $2d-2$ 
		precolored edges,
		respectively,
		and at most $2d-1$ colors, say $1,\dots, 2d-1$, are used by $\varphi$. 
		By the induction hypothesis, the restrictions of 
		$\varphi$ to $Q_1$ to $Q_i$ are extendable
		to $(2d-1)$-edge colorings $f_1$ and $f_i$, respectively, 
		using colors $1,\dots, 2d-1$.
		We color all other $Q_j$'s correspondingly to
		how $Q_i$ is colored, and
		using Lemma \ref{lem:oddcycle} we obtain an extension of $\varphi$.
		
		On the other hand, if both $Q_1$ and $Q_i$ contain at most $2d-2$
		precolored edges, respectively, but in total $2d$ colors $1,\dots, 2d$
		are used by $\varphi$, then every color appears on exactly
		one edge under $\varphi$. Hence, we may assume that 
		one edge of $Q_1$, but not $Q_i$, 
		is colored, say, $1$, and similarly, one edge of $Q_i$ is colored $2d$.
		By the induction hypothesis, there is an extension $f_1$
		of the restriction of $\varphi$
		to $Q_1$ using colors $1,\dots, 2d-1$, and an extension $f_i$ of 
		the restriction of
		$\varphi$ to $Q_i$ using colors $2,\dots, 2d$.
		
		Now, either $i \neq 2$ or $i \neq 2k+1$; suppose that the former holds.
		We define a proper edge coloring $f_2$ of $Q_2$ 
		using colors $2,\dots, 2d$ by coloring
		$Q_2$ correspondingly to $Q_1$ but using color $2d$ instead of $1$, 
		and then coloring
		all other planes of $G-E(D_1)$ correspondingly to how $Q_i$ is colored.
		By the construction of $f_2$, for each layer edge $e$ of $D_1$ between
		$Q_1$ and $Q_2$,
		there is a color in $\{2,\dots, 2d\}$ that does not appear at an endpoint of $e$.
		We color every such layer edge by this color, and then color the edges of every
		cycle in $D_1$ by colors $1$ and $2d+1$ alternately, and starting with color $1$
		at $Q_2$. This yields an extension of $\varphi$.

		\bigskip
		
	{\bf Case 3.} {\em All precolored edges are contained in at least three planes:}

	Let $Q_{j_1}, Q_{j_2}, \dots, Q_{j_s}$ be the planes of $G-E(D_1)$ that contain
	precolored edges, where $j_1 \leq j_2 \leq \dots \leq j_s \leq 2k+1$. Note that
	any two planes contain precolored edges of altogether at most $2d-1$ colors,
	and that there are two planes $Q_{j_i}$ and $Q_{j_{i+1}}$ that contain precolored
	edges of altogether at most $2d-2$ colors. We assume that 
	$Q_{j_1}$ and $Q_{j_s}$ are two such planes.
		
	Consider an arbitrary cycle $C$ in $D_1$. We partition the edges of $C$
	into paths $P_{12},\dots, P_{(s-1)s}, P_{s1}$ where $P_{r(r+1)}$ has its endpoints
	in $Q_{j_r}$ and $Q_{j_{r+1}}$. Now, for each path $P_{r(r+1)}$, 
	there are two colors
	$c_{r(r+1)}, c'_{r(r+1)} \in \{1,\dots, 2d+1\}$  
	so that none of these colors appear 
	in the restriction of $\varphi$
	to $Q_{j_r} \cup Q_{j_{r+1}}$. For $r=1,\dots s-1$,
	we color each path $P_{r(r+1)}$ alternately by colors
	$c_{r(r+1)}$ and $c'_{r(r+1)}$,
	so that the resulting edge coloring is proper.
	Now, by assumption we have that $Q_{j_1}$ and $Q_{j_s}$ contain  edges
	of altogether at most $2d-2$ colors.
	Hence, there are two colors $c$ and $c'$ that do not appear on edges in $Q_{j_1}$ 
	or $Q_{j_s}$, nor on an edge of $D_1$ that is incident with $Q_{j_1}$.
	We color the edges in the path of $C$ from $Q_{j_s}$ to $Q_{j_1}$ by colors
	$c$ and $c'$ so that the resulting coloring is proper.
		
	Next, we color all uncolored edges of $D_1$ correspondingly to how $C$ is colored.
	Now, each $Q_j$ is incident with
	edges of $D_1$ of two colors that do not appear on 
	edges of $Q_j$ under $\varphi$, and,
	moreover, each $Q_j$ contains at most $2d-2$ precolored edges.
	Hence, by the induction
	hypothesis, the restriction of $\varphi$ to each $Q_j$ can be extended to a proper
		edge coloring using colors that do not appear on edges of $D_1$ that are 
		incident with $Q_j$. In conclusion, $\varphi$ is extendable.
	\end{proof}
	
	
	\begin{lemma}
	\label{cl:ODDnoempty2free}
		If there is a dimension of $G$ with at most two precolored edges, 
		the colors of which do not appear
		on edges in any other dimension of $G$,
		then $\varphi$ is extendable.
	\end{lemma}
		\begin{proof}
		Let $D_1$ be a dimension containing at most two precolored edges, the colors
		of which do not appear on any edges in $G-E(D_1)$. 
		
		We first consider the case when only one color $c_1$ appears on the precolored edges of $D_1$.
		
		\bigskip
		
		{\bf Case 1.} {\em Only one color $c_1$ appears on the precolored edge(s) of $D_1$:}
		
		In this case the argument breaks into several subcases.
		
		\bigskip
		
		{\bf Case 1.1.} {\em All precolored edges of $G-E(D_1)$ are contained in one plane:}
		
		Let $Q_1$ be a plane in $G-E(D_1)$ 
		containing all precolored edges except the ones of $D_1$. 
		As before, there is an extension of the restriction of $\varphi$ to $Q_1$
		using colors $\{1,\dots, 2d+1\} \setminus \{c_1\}$ (by removing the colors
		of edges colored by some color $c_2\neq c_1$, taking an extension of the 
		resulting precoloring of $Q_1$ using colors 
		$\{1,\dots, 2d+1\} \setminus \{c_1,c_2\}$ 
		and then recoloring the edges that are $\varphi$-colored $c_2$).
		Next, we color all other planes of $G-E(D_1)$ correspondingly. 
		Now, since all planes of $G-E(D_1)$ are colored correspondingly,
		every edge of $D_1$ is adjacent to edges of $2d-2$
		different colors, so by Lemma \ref{lem:oddcycle},
		$\varphi$ is extendable.
		
		\bigskip
		
		{\bf Case 1.2.} {\em All precolored edges of $G-E(D_1)$ are contained in two planes:}

	Since at most two edges of $D_1$ are precolored, and only two planes
	contain precolored edges, there are two planes $Q_j$ and $Q_{j+1}$, 
	at most one of which contains precolored edges, and 
	such that there is no precolored edge between $Q_j$ and $Q_{j+1}$. Suppose e.g.
	$Q_{j+1}$ contains no precolored edges. Let 
	$c_2 \in \{1,\dots, 2d+1\}$ be a color such that no edge of 
	$G$ is precolored $c_2$.
	We take an extension of the restriction of $\varphi$ to $Q_j$ using colors
	$\{1,\dots, 2d+1\} \setminus \{c_1,c_2\}$, color $Q_{j+1}$ correspondingly, 
	and then color all edges between $Q_j$ and $Q_{j+1}$ by the unique color in
	$\{1,\dots, 2d+1\} \setminus \{c_1,c_2\}$ missing at its endpoints.
	Now, unless there are two precolored edges of 
	$D_1$ that are  contained in the same layer $P$
	and at even distance in the path $P'$ obtained from $P$ by removing
	the edge between $Q_j$ and $Q_{j+1}$, 
	we can color all edges of $D_1$ alternately by colors 
	$c_1$ and $c_2$, and then color all remaining planes of $Q-E(D_1)$ by 
	colors in $\{1,\dots, 2d+1\} \setminus \{c_1,c_2\}$ so 
	that the resulting edge coloring is proper and agrees with $\varphi$.
				
	Alternatively, if the distance between the two precolored edges of $D_1$ is
	even (in $P'$), then we select two additional planes 
	$Q_r$ and $Q_{r+1}$,
	containing no precolored edges between them and such 
	that at most one of $Q_r$ and
	$Q_{r+1}$ contains precolored edges. We may then repeat 
	the above coloring procedure for
	$Q_r$ and $Q_{r+1}$; we leave the details to the reader.

		\bigskip
		
		{\bf Case 1.3.} {\em All precolored edges of $G-E(D_1)$
		are contained in at least three planes:}

	Suppose first that there is only one precolored $e$ of $G-E(D_1)$,
	and let $Q_{j_1}, Q_{j_2}, \dots, Q_{j_s}$ be the planes of $G-E(D_1)$ that contain
	precolored edges, where $j_1 \leq j_2 \leq \dots \leq j_s$. Note that
	any two planes contain precolored edges of altogether at most $2d-2$ colors.
		Now as in Case 3 of the proof of the preceding lemma, we color the edges
		of the paths between pairs of planes with precolored edges by picking two colors
		that do not appear in the restrictions of $\varphi$ to these planes. 
		Naturally,
		we pick these colors so that in the path containing $e$, 
		the resulting coloring agrees with $\varphi$.
		Thereafter,
		we take extensions of the restrictions of $\varphi$ to the planes
		$Q_{j_1}, Q_{j_2}, \dots, Q_{j_s}$, so that the resulting coloring is proper.
		Hence, $\varphi$ is extendable.

		Suppose now that
		$G-E(D_1)$ contains two precolored edges $e_1$ and $e_2$.
		Then there are colors $c_2, c_3$ that do not appear on any edges of $G$ under $\varphi$.
		
		Now, if $e_1$ and $e_2$
		are corresponding edges or are not incident with a common plane, then we may proceed
		as in the preceding paragraph, but possibly pick
		three colors when coloring the paths between planes with precolored edges
		to ensure that the obtained coloring of $D_1$ is proper and agrees with $\varphi$.
		This is possible
		since any two planes contain at most $2d-3$ precolored edges.
		
		It remains to consider the case when $e_1$ and $e_2$ are incident with exactly
		one common plane $Q_1$. Suppose that $e_1$ in addition is incident with $Q_2$.
		If 
		there are at most $2d-4$ precolored edges in $Q_1 \cup Q_2$ 
		and at most $2d-4$ precolored edges in
		$Q_{2k+1} \cup Q_1$,
		then there are
		independent edges $e'_1$ and $e'_2$ in $Q_1$,
		adjacent to $e_1$ and $e_2$, respectively,
		and such that neither
		these edges, nor the corresponding edges of $Q_2$
		and $Q_{2k+1}$, respectively,
		are precolored.
		From the restriction of $\varphi$ to $Q_{2k+1} \cup Q_1 \cup Q_2$, we define a
		precoloring $\varphi'$ by coloring all these four edges of $Q_{2k+1} \cup Q_1 \cup Q_2$
		by the color $c_1$. We may now obtain an extension of $\varphi'$ by proceeding
		as in Case 3 of the preceding lemma, and thereafter swap colors on two bicolored $4$-cycles
		containing $e_1$ and $e_2$, respectively, to obtain an extension of $\varphi$.

		Suppose now instead that  $Q_1 \cup Q_{2k+1}$, say, contain exactly $2d-3$
		precolored edges. If there exist 
		independent edges $e'_1$ and $e'_2$ in $Q_1$,
		as described in the preceding paragraph, then we may proceed as in that case,
		so suppose that there are no two such edges.
		
		Then, since both $Q_{2k+1} \cup Q_1$ and $Q_1 \cup Q_2$ contain at most
		$2d-3$ precolored edges and every $Q_j$ is $(2d-2)$-regular, 
		the endpoints of $e_1$ and $e_2$ in $Q_1$ must be adjacent. Now,
		it is easy to see that this implies that there is either an edge $e'_1$ adjacent to $e_1$
		but not to $e_2$, such that $e'_1$ and the corresponding edge of $Q_2$ are not precolored,
		or an uncolored edge $e'_2$ adjacent to $e_2$ but not to $e_1$, and such that
		$e'_2$ and the corresponding edge of $Q_{2k+1}$ are not precolored.
		Suppose, for instance, that such an edge $e'_1$ exists.
	
		Consider the precoloring $\varphi'$
		obtained from the restriction of $\varphi$ to $Q_{2k+1} \cup Q_1 \cup Q_2$
		by in addition coloring $e'_1$ and also the corresponding edge of $Q_2$
		by the color $c_1$. Now, since there is no uncolored edge $e'_2$ as described above, it follows that
		all $d-2$ edges $a_1, \dots, a_{d-2}$ adjacent to $e_2$ in $Q_1$ satisfy
		that either $a_i$, or the corresponding edge of $Q_{2k+1}$, is $\varphi'$-precolored
		or adjacent to an edge colored $c_1$ under $\varphi'$.
		Moreover, since $Q_{2k+1} \cup Q_1$ is triangle-free and contains at most $2d-2$ $\varphi'$-precolored
		edges, every precolored edge of $Q_{2k+1} \cup Q_1$ satisfies this condition.
		Thus by properly coloring the uncolored edges adjacent to $e_2$,
		except the one adjacent to $e_1$, by colors from
		$\{1,\dots, 2d+1\} \setminus \{c_1, c_2,c_3\}$, we obtain a precoloring $\varphi''$
		from $\varphi'$. 
		Then every plane in $G-E(D_1)$ contains at most $2d-2$ precolored edges under $\varphi''$.
		Furthermore any extension of the restriction of $\varphi''$ to $Q_{2k+1} \cup Q_1$
		using colors $\{1,\dots, 2d+1\} \setminus \{c_2,c_3\}$ does not use $c_1$ on an edge
		adjacent to $e_2$. Now, since there is exactly one precolored edge
		of $G-E(D_1)$ that is not contained in $Q_{2k+1} \cup Q_1$, 
		we may once again proceed as in Case 3 of 
		Lemma \ref{cl:ODDoneempty} and color the edges of $D_1$ appropriately
		to obtain
		an extension of $\varphi''$ where no edge adjacent to $e_2$ is colored $c_1$. 
		Thereafter we may swap colors on a
		bicolored $4$-cycle and recolor $e_2$ to obtain an extension of $\varphi$.
		
		
		\bigskip
		
		{\bf Case 2.} {\em The precolored edges of $D_1$ are colored differently:}
		
		Suppose now that $D_1$ contains two precolored edges, colored $c_1$ and $c_2$, respectively,
		and that $c_3$ is a color that does not appear on any edge under $\varphi$.
		If there is an extension of the restriction of $\varphi$ to $D_1$ using colors $\{c_1,c_2,c_3\}$,
		such that all edges of $D_1$ are colored correspondingly, then there are extensions
		of the restrictions of $\varphi$ to all the planes $G-E(D_1)$ using colors that do not appear
		on incident edges of $D_1$. Hence, $\varphi$ is extendable. 
		
		On the other hand, if there is no such extension of the restriction of $\varphi$, then
the two precolored edges $e_1$ and $e_2$ of $D_1$ are incident with the same pair of planes, say $Q_1$ and $Q_2$. Now, if $Q_1 \cup Q_2$ contains at most $2d-4$ precolored edges,
		then there are uncolored corresponding edges $e'_1 \in E(Q_1)$ and 
		$e'_2 \in E(Q_2)$ that are adjacent to $e_2$, but not to $e_1$.
		We may now color these edges $c_2$ and remove the color from $e_2$
		to obtain the precoloring $\varphi'$ from $\varphi$, 
		and then proceed as in Case 3 when only one edge of $D_1$ is precolored to obtain
		an extension of $\varphi'$. Thereafter we swap colors on a 
		bicolored $4$-cycle to obtain an extension 
		of $\varphi$.
		
		If, on the other hand, $Q_1 \cup Q_2$ contains at least $2d-3$ precolored edges,
		then there is at most one edge in $G-E(D_1) \cup E(Q_1) \cup E(Q_2)$  that is precolored.
		Without loss of generality, we assume that $Q_3$ contains no precolored edge.
		By the induction hypothesis, there is an extension of the restriction of $\varphi$ to $Q_1 \cup Q_2$
		using colors $\{1,\dots, 2d+1\} \setminus \{c_1,c_2\}$. We color $Q_3$ correspondingly to
		how $Q_2$ is colored, and every edge between $Q_2$ and $Q_3$ by the color in
		 $\{1,\dots, 2d+1\} \setminus \{c_1,c_2\}$ missing at its endpoints. All other edges in $D_1$
		 are colored $c_1, c_2$ alternately so that the coloring agrees with $\varphi$.
		 Finally, we color all hitherto uncolored planes using colors $\{1,\dots, 2d+1\} \setminus \{c_1,c_2\}$
		 so that the resulting coloring agrees with $\varphi$. In conclusion, $\varphi$ is extendable.
		\end{proof}


		\begin{lemma}
	\label{cl:ODDnofree}
		If each dimension of $G$ contains precolored edges and there is a dimension
		with exactly one precolored edge, the color of which does appear on
		edges in other dimensions of $G$,
		then $\varphi$ is extendable.
	\end{lemma}
	\begin{proof}
			Let $D_1$ be a dimension containing only one precolored
		edge $e$, colored, say $c_1$, and consider the 
		subgraph $G-E(D_1)$ consisting of $2k+1$ planes
		$Q_1, \dots, Q_{2k+1}$.
		Since at least one color appears on at least two edges, there are
		two colors $c_2,c_3 \in \{1,\dots, 2d+1\}$ that do not appear
		on any edge under $\varphi$.
	
			\bigskip

		{\bf Case 1.} {\em All precolored edges of $G-E(D_1)$ are contained in one plane:}
		
		In this case, we may proceed as in Case 1 of 
		Lemma \ref{cl:noemptynofree}, but use
		Lemma \ref{lem:oddcycle} instead of Lemma \ref{lem:evencycle}. We omit the details.

		\bigskip
		
		{\bf Case 2.} {\em All precolored edges of $G-E(D_1)$ are contained in two planes:}

		Let $Q_1$ and $Q_i$ be the two planes containing the precolored edges
		distinct from $e$.
		
		Let us first consider the case when $e$ is not incident to $Q_1$ or $Q_i$. 
		Since each of $Q_1$ and $Q_i$ contains at most $2d-2$ precolored edges,
		there are extensions $f_1$ and $f_i$ of the restrictions of $\varphi$
		to $Q_1$ and $Q_i$, respectively, using colors $\{1, \dots, 2d+1\}
		\setminus \{c_2,c_3\}$. Therafter we color the edges of $D_1$ properly and correspondingly
		using colors $\{c_1,c_2,c_3\}$ so that the coloring agrees with the restriction
		of $\varphi$ to $D_1$ and has no conflicts with $f_1$ or $f_i$. Finally
		we color the remaining planes of $G-E(D_1)$, as to obtain an extension of $\varphi$.
	
		\bigskip

		Suppose now that $e$ is incident with $Q_1$ and $Q_2$, and $i \neq 2$.
		Then either $Q_3$ or $Q_{2k+1}$ contains no precolored edges;
		suppose $Q_3$. (The case when $Q_{2k+1}$
		has this property is similar.)
		As in the preceding paragraph,
		there are extensions $f_1$ and $f_i$ of the restrictions of $\varphi$
		to $Q_1$ and $Q_i$, respectively, using colors $\{1, \dots, 2d+1\}
		\setminus \{c_2,c_3\}$. We color $Q_2$ and $Q_3$ correspondingly to how
		$Q_1$ is colored.
		Next, we color every edge of $D_1$ between $Q_2$ and $Q_3$
		by a color in $\{1,\dots, 2d+1\} \setminus \{c_2, c_3\}$ missing at its endpoints,
		and then color all other edges of $D_1$ alternately by colors $c_2$ and $c_3$ so
		that all edges between $Q_1$ and $Q_2$ are colored $c_2$. The remaining
		uncolored planes of $G$ are properly colored using colors
		$\{1,\dots, 2d+1\} \setminus \{c_2, c_3\}$.
		Now, if $e$ is adjacent to edges colored $c_1$, 
		then we swap colors on a bicolored $4$-cycle containing $e$ and two edges colored
		$c_1$ to obtain an extension of $\varphi$; otherwise we simply recolor $e$
		to obtain an extension of $\varphi$.
		
		\bigskip
		
		Suppose now that $e$ is incident with $Q_1$ and $Q_2$, and $i = 2$.
		By removing the color from any edge that is colored $c_1$ under $\varphi$,
		we obtain a precoloring $\varphi'$ of $G$. The restriction of
		$\varphi'$ to $Q_1$ and $Q_2$ are extendable to proper edge colorings,
		respectively, using colors $\{1,\dots, 2d+1\} \setminus \{c_1,c_2\}$.
		By recoloring any edge of $Q_1$ and $Q_2$ that is $\varphi$-colored 
		$c_1$ by the color $c_1$
		we obtain edge colorings $f_1$ of $Q_1$ and $f_2$ of $Q_2$, respectively.

		Next, we color every edge between $Q_1$ and $Q_2$ by the color $c_2$
		except that $e$ is colored $c_1$. Thereafter, we color
		$Q_3, \dots, Q_{2k}$ correspondingly to how $Q_2$ is colored, 
		and $Q_{2k+1}$ correspondingly
		to how $Q_1$ is colored. Now, 
		for every vertex $x$ of $Q_2$, there are colors $c_x, c'_x \in \{1,\dots, 2d+1\}$
		that do not appear at $x$ in $Q_2$ or on the incident edge between $Q_1$ and $Q_2$.
		We color every path in $D_1$ from $Q_2$ to $Q_{2k}$ by colors
		$c_x$ and $c'_x$ alternately, and thereafter color every edge between $Q_{2k}$
		and $Q_{2k+1}$ by the color of the edge in the same layer between $Q_1$ and $Q_2$.
		Finally, we color the edges between $Q_1$ and $Q_{2k+1}$ by a color missing at its
		endpoints to obtain an extension of $\varphi$.

		\bigskip
		
		{\bf Case 3.} {\em All precolored edges of $G-E(D_1)$ are contained in at least three planes:}

	The assumption implies that every plane in $G-E(D_1)$ contains at most
	$2d-3$ precolored edges. Assume that $e$ is incident with $Q_1$ and $Q_2$.
		
	Suppose first that $Q_1 \cup Q_2$ contains altogether at most $2d-3$ precolored
		edges. Then there are uncolored corresponding
		edges $e_1 \in E(Q_1)$ and $e_2 \in E(Q_2)$ that are adjacent to $e$ but not to
		any edge in $Q_1 \cup Q_2$ $\varphi$-colored $c_1$. From the restriction of
		$\varphi$ to $G-E(D_1)$ we define a new precoloring 
		$\varphi'$ by coloring $e_1$
		and $e_2$ by the color $c_1$.
		Thereafter we may proceed as in Case 3 of the proof of
		Lemma \ref{cl:ODDoneempty} to obtain a proper $(2d+1)$-edge coloring of
		$G$ which is an extension of $\varphi'$ and where the edges of $D_1$ are colored
		correspondingly.
		Thus, by swapping colors on a bicolored
		$4$-cycle we obtain an extension of $\varphi$.
		
		Suppose now that $Q_1 \cup Q_2$ contains altogether $2d-2$ precolored
		edges. If there are uncolored corresponding
		edges $e_1 \in E(Q_1)$ and $e_2 \in E(Q_2)$ that are adjacent to $e$ but not to
		any edge in $Q_1 \cup Q_2$ $\varphi$-colored $c_1$, then we proceed
		as in the preceding paragraph. So assume that there are no such edges
		$e_1$ and $e_2$. Then there are $2d-2$ edges $e_1,\dots, e_{2d-2}$ in $Q_1$
		that are adjacent to $e$ and such that each of these edges
		satisfies that
		\begin{itemize}
		
		\item[(i)] $e_j$ or the corresponding edge of $Q_2$ is precolored by a color distinct 
		from $c_1$, or
		
		\item[(ii)] $e_j$ or the corresponding edge of $Q_2$ is adjacent to an edge
		precolored $c_1$.
		
		\end{itemize}
		Moreover, since $Q_1 \cup Q_2$ is triangle-free and
		contains at most $2d-2$ precolored edges,
		every precolored edge in $Q_1 \cup Q_2$ satisfies one of these conditions. Now,
		for $j=1,2$, from the restriction of $\varphi$ to $Q_j$,
		we define a new precoloring $\varphi_j$ of $Q_j$
		by coloring every edge of $Q_1$ and $Q_2$ that is adjacent
		to $e$ and does not satisfy
		(i) or (ii) by a color in $\{1,\dots, 2d+1\} \setminus\{c_1,c_2,c_3\}$ so that the
		resulting coloring is proper and agrees with $\varphi$. Now, by the induction 
		hypothesis, $\varphi_j$ is extendable to a proper edge coloring of $Q_j$ using
		colors $\{1,\dots, 2d+1\} \setminus\{c_2,c_3\}$. Note that no edge of 
		$Q_1$ or $Q_2$
		adjacent to $e$ is colored $c_1$ in these colorings. 
		Thus we may color all edges between
		$Q_1$ and $Q_2$ by $c_2$ except $e$ which is colored $c_1$.
		
		Next, suppose that $Q_r$, $r\notin\{1,2\}$, is the third plane containing a precolored
		edge. Then $Q_{r+1}$ or $Q_{r-1}$ contains no precolored or hitherto colored
		edges, suppose $Q_{r+1}$.
		We take an extension of the restriction of $\varphi$ to $Q_r$ 
		using colors $\{1,\dots, 2d+1\} \setminus\{c_2,c_3\}$, and color all other uncolored
		planes correspondingly to how $Q_r$ is colored. Thereafter we color the edges
		between $Q_r$ and $Q_{r+1}$ by the unique color in
		$\{1,\dots, 2d+1\} \setminus\{c_2,c_3\}$ missing at it endpoints.
		Finally, we color all remaining uncolored edges of $D_1$ by colors 
		$c_2,c_3$ alternately so that the resulting coloring is proper.
		This yields an extension of $\varphi$.
	\end{proof}


		\begin{lemma}
	\label{cl:ODD2edges}
		If each dimension of $G$ contains exactly two precolored edges, 
		at least one of which is colored
		by a color appearing on precolored edges in other dimensions, then $\varphi$
		is extendable.
	\end{lemma}
	
	\begin{proof}
	Let $D_1$ be a dimension containing two precolored 
	edges $e_1$ and $e_2$.
	By assumption,
	at most $2d-1$ colors appear on edges under $\varphi$, so let let $c_3, c_4$
	be two colors from $\{1,\dots, 2d+1\}$ that do not appear on any edges
	under $\varphi$.

	\bigskip
	
	{\bf Case 1.} {\em All precolored edges of $G-E(D_1)$ are 
	contained in one plane:}
	
	Suppose that all precolored edges except $e_1$ and $e_2$ lie in
	one component $Q_1$ of $G-E(D_1)$.
	Without loss of generality, we assume that 
	$\{\varphi(e_1), \varphi(e_2)\} \subseteq \{c_1,c_2\}$. We define a new precoloring
	$\varphi'$ from the restriction of $\varphi$ to $Q_1$ by removing the colors $c_1$
	and $c_2$ from any edges of $Q_1$ colored by these colors. Now, by the induction 
	hypothesis $\varphi'$ is extendable to a proper $(2d-1)$-edge 
	coloring of $Q_1$ using
	colors $\{1,\dots, 2d+1\} \setminus \{c_1,c_2\}$. By recoloring the edges of $Q_1$
	that are $\varphi$-colored $c_1$ and $c_2$ by colors $c_1$ and $c_2$, respectively,
	we obtain a proper edge coloring $f'$ of $Q_1$.
	Next, we color all other planes of $G-E(D_1)$ correspondingly and
	define a list assignment for the edges of $D_1$ by assigning every edge
	the set of colors from $\{1,\dots, 2d+1\}$ not appearing on its adjacent edges. Then
	each edge of
	$D_1$ receives a list of $3$ colors except for the two edges $e_1$ and $e_2$ that
	are precolored. Thus, there is an extension of $\varphi$.
		
		\bigskip
		
	{\bf Case 2.} {\em All precolored edges of $G-E(D_1)$ are contained in two planes:}
		
			We shall consider several different subcases.
			
		\bigskip
		
		{\bf Case 2.1.} {\em The precolored edges of $D_1$ have the same 
		color under $\varphi$:}
		
		Suppose that $\varphi(e_1) = \varphi(e_2) = c_1$. 
		If $e_1$ and $e_2$ are both incident with the same two planes,
		then we may apply arguments which are similar
		to the ones in Case 2 of the proof of Lemma \ref{cl:ODDnofree}.
		Consequently, assume that $e_1$ and $e_2$ are incident with at
		most one common plane.
		
		\bigskip
		
		{\bf Subcase 2.1.1.} {\em $e_1$ and $e_2$ are both incident with
		exactly one common plane:}
		
		Suppose that $e_1$ and $e_2$ are both incident with 
		the common plane $Q_1$, and that $e_1$ is also incident with $Q_2$.
		 If $Q_1$ and $Q_2$ contain all precolored edges of $G-E(D_1)$, then
		a similar argument as in the subcase of Case 2 of Lemma \ref{cl:ODDnofree} when $Q_1 \cup Q_2$
		contains all precolored edges of $G-E(D_1)$ again applies, so we omit the details
		here as well.
		
		It remains to consider the following subcases:
		\begin{itemize}
		
			\item[(a)] $Q_1$ contains precolored edges, but neither of $Q_2$ and
			$Q_{2k+1}$.
		
			\item[(b)] Either $Q_{2k+1}$ or $Q_2$, but not $Q_1$, contains precolored edges.
			
			\item[(c)] Both $Q_2$ and $Q_{2k+1}$ contain precolored edges.
		
		\end{itemize}
		{\em (a) holds:}
		
		Suppose that $Q_1$ and $Q_i$ contain all precolored edges of $G-E(D_1)$,
		where $i\notin \{1,2,3,2k+1\}$.
		If $e_1$ and $e_2$ are adjacent via an uncolored 
		edge $e$ in $Q_1$, then since
		$Q_1$ contains at most $2d-3$ precolored edges, there is a color 
		$c \in \{1,\dots, 2d+1\} \setminus \{c_1,c_3,c_4\}$ 
		that does not appear on any edge adjacent to $e$.
		Thus the precoloring $\varphi'$
		obtained from $\varphi$ by in addition coloring $e$ by the color $c$
		is proper.  On the other hand, 
		if there is no such edge, then we set $\varphi' = \varphi$.
		
		Now, since both $Q_1$ and $Q_i$ contain at most $2d-2$ $\varphi'$-precolored edges, there
		are extensions $f_1$ and $f_i$, respectively,
		of the restrictions of $\varphi'$ to $Q_1$ and $Q_i$, respectively, using
		colors $\{1,\dots, 2d+1\} \setminus \{c_3, c_4\}$. We color
		$Q_2, Q_3, Q_{2k+1}$ correspondingly to how $Q_1$ is colored, and thereafter 
		color every edge between $Q_2$ and $Q_3$ by the color in
		$\{1,\dots, 2d+1\} \setminus \{c_3, c_4\}$ missing at its endpoints. 
		
		Next, we color all hitherto 
		uncolored edges of $D_1$ alternately using colors $c_3, c_4$
		so that all edges between $Q_1$ and $Q_2$ have color $c_3$, 
		and also color all hitherto uncolored planes properly using colors
		$\{1,\dots, 2d+1\} \setminus \{c_3, c_4\}$. The obtained coloring
		is proper and agrees with $\varphi'$ except for $e_1$ and $e_2$.
		Now, if neither $e_1$ of $e_2$ are adjacent to an edge colored $c_1$, then we simply
		recolor them; otherwise, we swap on one or two bicolored cycles to
		obtain an extension of $\varphi$;
		note that if both $e_1$ and $e_2$ are adjacent to edges colored $c_1$, then
		these cycles are disjoint. 
		Hence, $\varphi$
		is extendable.
		
		\bigskip
		
		{\em (b) holds:}
		
		If instead either $Q_{2k+1}$ or $Q_2$, but not $Q_1$, contains precolored edges,
		then a similar argument as in (a) applies, so we omit the details.

		\bigskip
		
		{\em (c) holds:}
		
		Assume that $Q_2$ and $Q_{2k+1}$ contain all precolored edges of $G-E(D_1)$, and let
		$u_{2k+1}$ and $u_2$ be the vertices of $Q_{2k+1}$ and $Q_{2}$
		that are incident with $e_1$ and $e_2$, respectively.
		
		If both $Q_{2k+1}$ and $Q_2$ contain at most $2d-4$ precolored edges,
		then there are uncolored edges $e'_{2k+1} \in E(Q_{2k+1})$ 
		and $e'_{2} \in E(Q_{2})$ that are incident
		with $u_{2k+1}$ and $u_2$, respectively, not adjacent to any edges
		of $Q_{2k+1} \cup Q_2$ precolored $c_1$, and
		such that the corresponding edges of $Q_1$
		are independent.
		We color $e'_{2k+1}$, $e'_{2}$, and also the corresponding edges
		of $Q_1$ by the color $c_1$. Together with $\varphi$, this defines a precoloring
		of $Q_{2k+1} \cup Q_1 \cup Q_2$, which by the induction hypothesis is
		extendable to a proper edge coloring using colors 
		$\{1,\dots, 2d+1\} \setminus \{c_3,c_4\}$. We color all edges between $Q_1$ and $Q_2$
		with color $c_4$, and all edges between $Q_{2k+1}$ and $Q_1$ by the color $c_3$.
		Next, we color $Q_{2k}$ and $Q_3$ correspondingly to how $Q_{2k+1}$ and $Q_2$
		are colored, respectively. Thereafter we color all edges between $Q_{2k+1}$
		and $Q_{2k}$ by the color in $\{1,\dots, 2d+1\} \setminus \{c_3,c_4\}$ missing
		at its endpoints, and similarly for $Q_2$ and $Q_3$.
	Finally, we color all uncolored edges of $D_1$ alternately by colors $c_3, c_4$,
	color all hitherto uncolored planes using colors 
	$\{1,\dots, 2d+1\} \setminus \{c_3, c_4\}$ and swap on two 
	bicolored cycles containing
	$e_1$ and $e_2$, respectively, to obtain an extension of $\varphi$.

	Suppose now instead that one of $Q_{2k+1}$ and $Q_2$ contains 
	$2d-3$ precolored edges, say $Q_{2k+1}$. 
	Then $Q_{2}$ contains exactly one precolored edge $e$. 
	By removing the color from any
		edge of $Q_{2k+1}$ that is precolored $c_1$, we obtain a precoloring $\varphi'$ from
		the restriction of $\varphi$ to $Q_{2k+1}$. 
		$\varphi'$ is extendable to a proper coloring of $Q_{2k+1}$ using colors 
		$\{1,\dots, 2d+1\} \setminus \{c_1,c_3\}$ and by recoloring the edges of $Q_{2k+1}$
		$\varphi$-colored $c_1$ by the color $c_1$ we obtain an extension $f_{2k+1}$ of 
		the restriction of $\varphi$ to $Q_{2k+1}$.
		
		Next, we color $Q_1$ correspondingly to $Q_{2k+1}$ except that we color any edge
		of $Q_1$ corresponding to an edge colored $c_1$ by the color $c_3$.
		We color the edges between $Q_{2k+1}$ and $Q_1$ by an arbitrary color 
		in $\{1,\dots, 2d+1\} \setminus \{c_1,c_3\}$
		not appearing at its endpoints, except that $e_2$ is colored $c_1$.
		
		Suppose first that $e$ is precolored $c_1$. Then we color $Q_2$
		correspondingly to how $Q_1$ is colored, but color $e$ by color $c_1$.
		Next we color all edges between $Q_1$ and $Q_2$ by an arbitrary color
		in $\{1,\dots, 2d+1\}$
		missing at its endpoints except that $e_1$ is colored $c_1$.
		Thereafter, we color $Q_3$ correspondingly to how $Q_1$ is colored,
		and all remaining uncolored
		planes correspondingly to how $Q_{2k+1}$ is colored. We may then
		color the hitherto uncolored edges of $D_1$ appropriately to obtain an extension
		of $\varphi$.

		Suppose now that the precolored edge of $Q_2$ is colored $c_2 \neq c_1$.
		Let $e'$ be the edge of $Q_1$ corresponding to $e$, and assume that $e'$
		is colored $c'$ in the hiherto constructed coloring.
		We color $Q_2$ correspondingly to how $Q_1$ is colored but permute the colors
		$c_2$ and $c'$ in the coloring of $Q_2$. Thereafter, we color $Q_3$ correspondingly to $Q_2$
		except that we permute colors $c_2$ and $c'$, and finally we color the remaining
		uncolored edges of $G$ by proceeding as in the preceding paragraph.

		\bigskip
		
		\bigskip
		
		{\bf Subcase 2.1.2} {\em $e_1$ and $e_2$ are not incident with
		a common plane:}

		Suppose that $e_1$ is incident with $Q_1$ and $Q_2$
		and $e_2$ is incident with $Q_j$ and $Q_{j+1}$, and all these four planes are
		distinct.
		If all precolored edges are contained in $Q_1 \cup Q_2$, then 
		as before we may then select corresponding uncolored edges $e'_j$ and $e'_{j+1}$
		of $Q_j$ and $Q_{j+1}$ that are adjacent to $e_2$.
		Next, we consider the precoloring $\varphi'$
		obtained from $\varphi$ by coloring $e'_j$ and $e'_{j+1}$ by $c_1$ and 
		removing the color
		$c_1$ from $e_2$. We may now apply similar arguments as in the subcase of Case 3 of
		Lemma \ref{cl:ODDnofree} when $Q_1 \cup Q_2$ contains $2d-2$ precolored edges
		to obtain an extension of $\varphi'$. In particular, 
		since there is at least one plane in
		$G$ that is distinct from $Q_1, Q_2, Q_j, Q_{j+1}$
		that contains no $\varphi'$-colored edges,
		we can
		color the edges of $D_1$ so that all edges between $Q_{j}$ and $Q_{j+1}$
		have the same color. We may then swap on a bicolored $4$-cycle to obtain an 
		extension
		of $\varphi$.

	Suppose now that exactly one of the planes $Q_1$ and $Q_2$, and exactly one of 
	the planes $Q_{j}$ and $Q_{j+1}$ contain precolored edges.
	Assume e.g. that $Q_1$ and $Q_j$ contain no precolored edges (the other cases
	are analogous). We pick an edge $e'_2$ in $Q_2$ that is uncolored 
	and adjacent to $e_1$, but not adjacent
	to any other edge precolored $c_1$, and a similar edge $e'_{j+1}$ of $Q_{j+1}$;
	since each of these planes contains at most $2d-3$ precolored edges, such edges
		exist.
		From the restriction of $\varphi$ to $Q_2 \cup Q_{j+1}$ we define a new precoloring
		$\varphi'$ by in addition coloring $e'_2$ and $e'_{j+1}$ $c_1$.
		Now, by the induction hypothesis, there are extensions $f_2$ and $f_{j+1}$ 
		of the restrictions of $\varphi'$ to $Q_2$ and $Q_{j+1}$, respectively,
		using colors $\{1,\dots, 2d+1\} \setminus \{c_3, c_4\}$.
		
		Let us now color the other planes of $G-E(D_1)$. Without loss of generality,
		we assume that $j+1 < 2k+1$. We color $Q_1$ and $Q_{2k+1}$ correspondingly to
		how $Q_2$ is colored, $Q_j$ correspondingly to how $Q_{j+1}$ is colored,
		and all other planes arbitrarily using colors
		$\{1,\dots, 2d+1\} \setminus \{c_3, c_4\}$.
		Thereafter we color all edges between $Q_1$ and $Q_{2k+1}$ by a color in
		$\{1,\dots, 2d+1\} \setminus \{c_3, c_4\}$ missing at its endpoints, and all other
		edges of $D_1$ alternately using colors $c_3,c_4$
		and starting with color $c_3$ at $Q_1$.
		Finally, we swap colors on two bicolored $4$-cycles containing $e_1$ and $e_2$,
		respectively, to obtain an extension of $\varphi$.
		
		Finally, we consider the case when $Q_1$ may contain precolored edges,
		but none of $Q_2, Q_j, Q_{j+1}$ contain precolored edges.
		We define a precoloring $\varphi'$ from the restriction of $\varphi$ to $Q_1$
		by selecting an edge $e'_1 \in E(Q_1)$ adjacent to $e_1$
		and coloring it $c_1$, as before. Thereafter, we take an extension of $\varphi'$
		using colors $\{1,\dots, 2d+1\} \setminus \{c_3, c_4\}$,
		and color $Q_2$ correspondingly. Next, we color all edges between 
		$Q_{j}$ and $Q_{j+1}$ by the color $c_1$, and all other edges of $D_1$ alternately
		using colors $c_3$ and $c_4$, and starting with color $c_3$ at $Q_{j+1}$.
		We now obtain an extension of $\varphi$ as before.

		\bigskip
		
		{\bf Case 2.2} {\em The precolored edges of $D_1$ are colored differently under
		$\varphi$:}
		
		Suppose that $\varphi(e_1) = c_1$ and $\varphi(e_2)=c_2$.
		We shall consider some different cases.
		
		\bigskip
		
		{\bf Subcase 2.2.1} {\em $e_1$ and $e_2$ are both incident with
		two common planes:}
		
		Suppose that $e_1$ and $e_2$ are both incident with the planes
		$Q_1$ and $Q_2$. 
		Let $u_1$ and $u_2$ be the vertices in $Q_1$ that are incident with
		$e_1$ and $e_2$, respectively. 
		
		If none of $Q_1$ and $Q_2$ contain precolored edges,
		then we can select independent edges in $Q_1$ (and $Q_2$)
		that are incident with $u_1$ and $u_2$, respectively, and color
		them $c_1$ and $c_2$. Then we may proceed as in Case 3
		of Lemma \ref{cl:ODDoneempty} to obtain an extension of the resulting precoloring
		$\varphi'$ of $G-E(D_1)$, and thereafter we  obtain an extension of $\varphi$,
		as before.
		
		\bigskip

		Let us now assume that all precolored edges are contained in $Q_1 \cup Q_2$.
		We first prove the following claim.
		
		\begin{claim}
		\label{cl:edge}
			Suppose $d \geq 3$.
			At least one of the following two statements hold.
			
			(i) There is an edge
			$e'_1 \in E(Q_1)$ incident with $u_1$ but not $u_2$, such that
			$e'_1$ and the corresponding edge $e''_1$ of $Q_2$ are uncolored
			and not adjacent to any edge colored $c_1$.
			
			(ii) There is an edge $e'_2 \in E(Q_1)$ incident with $u_2$ but not $u_1$,
			such that $e'_2$ and the corresponding edge $e''_2$ of $Q_2$ are uncolored
			and not adjacent to any edge colored $c_2$.
		\end{claim}
		\begin{proof}
			Suppose that (i) is false. Since $G-E(D_1)$ is triangle-free and $(2d-2)$-regular, 
			there are 
			$2d-3$ edges $a_1, \dots, a_{2d-3} \in E(Q_1)$ incident with $u_1$,
			all of which are either precolored, adjacent to an edge colored $c_1$,
			or satisfies that the corresponding edge of $Q_2$ satisfies one of these conditions.
			Since $Q_1 \cup Q_2$ contains $2d-2$ precolored edges, it is
			easy to see that then (ii) must hold, so
			there is an edge $e'_2$ as desired.
		\end{proof}
			If $d = 2$, we note that the claim might fail if $u_1$ and $u_2$
			are adjacent via an uncolored edge. However, in this case, it is trivial
			to verify that $\varphi$ is extendable, since every precoloring of an
			odd cycle is extendable using $3$ colors.
			
			Suppose now that  (ii) of Claim \ref{cl:edge} holds, and let $e'_2$ and $e''_2$
			be corresponding edges of $Q_1$ and $Q_2$ respectively, 
			as described in the claim.
			From the restriction of $\varphi$ to $Q_1 \cup Q_2$ 
			we define a new precoloring
			$\varphi'$ of $Q_1 \cup Q_2$ by in addition coloring $e'_2$ and 
			$e''_2$ by the color $c_2$. We may now proceed as in the subcase of Case 2 of 
			the proof of Lemma \ref{cl:ODDnofree} when all the precolored edges
			are contained in $Q_1 \cup Q_2$, to obtain an extension of $\varphi'$
			(with $c_3$ in place of $c_2$). Thereafter, we swap colors on a
			bicolored $4$-cycle containing $e_2$ to obtain an extension of $\varphi$.
			
				\bigskip
				
	It remains to consider the case when all precolored edges are contained 
	in $Q_1$ and $Q_i$,
	where $i \neq 2$. Then either $Q_3$ or $Q_{2k+1}$ contains no precolored
	edges, say $Q_3$. 
			
	Suppose first that $Q_1$ contains at most $2d-4$ precolored edges.
	Then since $Q_1$ is $(2r-2)$-regular,  there are independent uncolored 
	edges $e'_1 \in E(Q_1)$ and $e'_2 \in E(Q_1)$
	incident with $u_1$ and $u_2$, respectively, and such that
	$e'_1$ is not adjacent to any edge of $Q_1$ $\varphi$-colored $c_1$,
	and $e'_2$ is not adjacent to any edge of $Q_1$ $\varphi$-colored $c_2$.
			
	From the restriction of $\varphi$ to $Q_1$
	we define a new precoloring $\varphi'$ of $Q_1$ by in addition coloring $e'_1$
	by the color $c_1$, and $e'_2$ by the color $c_2$.
	Next, we take an extension of $\varphi'$ using colors 
	$\{1,\dots, 2d+1\} \setminus \{c_3,c_4\}$, and color $Q_2$ and $Q_3$
	correspondingly to how $Q_1$ is colored.
	Thereafter, we color the edges of $D_1$ as follows: color all edges between
	$Q_2$ and $Q_3$ by a color in $\{1,\dots, 2d+1\} \setminus \{c_3,c_4\}$
	missing at its endpoints, and color all other edges of $D_1$ alternately
	using colors $c_3$ and $c_4$ so that all edges between $Q_1$ and $Q_2$
	are colored $c_3$. Thereafter, we color the planes $Q_4,\dots, Q_{2k+1}$
	with colors $\{1,\dots, 2d+1\} \setminus \{c_3,c_4\}$
	so that the coloring agrees with $\varphi$. Now we may obtain an extension of
	$\varphi$ by swapping colors on two bicolored $4$-cycles containing
	$e_1$ and $e_2$, respectively.

	Suppose now that $Q_1$ contains exactly $2d-3$ precolored edges. 
	Then $Q_i$ contains exactly one precolored edge.
			Moreover, as in the proof of Claim \ref{cl:edge}
			it is straightforward that there is 
			\begin{itemize}
			
			\item either an edge $e'_1$ satisfying
			(i) of Claim \ref{cl:edge}, or
			
			\item an edge $e'_2$ satisfying
			(ii) of Claim \ref{cl:edge}.
			
			\end{itemize}
			Suppose e.g. that (i) holds. Then from the restriction of $\varphi$ to $Q_1$
			we define a new precoloring $\varphi'$ of $Q_1$ by in addition coloring $e'_1$
			by the color $c_1$ and removing the color from any edge of $Q_1$ that is
			colored $c_2$.
			
			Next, we take an extension of $\varphi'$ using colors 
			$\{1,\dots, 2d+1\} \setminus \{c_2,c_3\}$, recolor the edges $\varphi$-colored
			$c_2$ by the color $c_2$, and thereafter color $Q_2$ and $Q_3$
			correspondingly to how $Q_1$ is colored. Denote the obtained coloring by $f$.
			We color all edges between $Q_1$ and
			$Q_2$ by the color $c_3$ except $e_2$ which is colored $c_2$, and color
			the edges between $Q_2$ and
			$Q_3$ by a color in $\{1,\dots, 2d+1\} \setminus \{c_2,c_3\}$
			missing at its endpoints. 
			
	Now, let $e''_i$ be the edge of $Q_i$ that is precolored, and let 
	$e''_1$ be the corresponding edge of $Q_1$. If $f(e''_1) = \varphi(e''_i)$,
	then we color all hitherto uncolored 
	planes correspondingly to how $Q_1$ is colored, therafter
	color the remaining uncolored
	edges of $D_1$ and finally swap colors on a bicolored $4$-cycle containing $e_1$
	to obtain an extension of $\varphi$.
			
	Otherwise, if $f(e''_1) \neq \varphi(e''_i)$, then we color all other planes
	correspondingly to how $Q_1$ is colored, except that we permute the colors
	$f(e''_1)$ and $\varphi(e''_i)$ in the colorings. 
	We may now apply similar arguments as before 
	to obtain an extension of $\varphi$; we leave the details to the reader.

		\bigskip

			{\bf Subcase 2.2.2} {\em $e_1$ and $e_2$ are incident with
		exactly one common plane:}

		Suppose that $e_1$ is incident with $Q_1$ and $Q_2$, and $e_2$ with
		$Q_1$ and $Q_{2k+1}$. Let $u_1$ and $u_2$ be the vertices of $Q_1$
		that are incident with $e_1$ and $e_2$, respectively.
		If neither of $Q_1, Q_2, Q_{2k+1}$ contain precolored edges, then a similar
		argument as in the second paragraph of Subcase 2.2.1 applies.
		Thus it suffices to consider the following subcases:
		\begin{itemize}
			
	\item[(a)] All precolored edges of $G-E(D_1)$ are contained in $Q_1 \cup Q_2$.
			
	\item[(b)] $Q_1$ contains precolored edges, but neither of $Q_2$ and $Q_{2k+1}$.
			
	\item[(c)] $Q_2$, but not $Q_1$ or $Q_{2k+1}$, contains precolored edges.
			
	\item[(d)] All precolored edges of $G-E(D_1)$ are contained in $Q_2 \cup Q_{2k+1}$.
		
		\end{itemize}
		By symmetry, it suffices to consider these cases.
		
		\bigskip
		
		{\em (a) holds:}
		
		We first consider the case when there is an edge $e'_1 \in E(Q_1)$
		adjacent to $e_1$, such that both $e'_1$
		and the corresponding edge $e'_2 \in E(Q_2)$ are uncolored and not adjacent
		to any edge precolored $c_1$. 
		If this holds, then from the restriction of $\varphi$ to
		$Q_1 \cup Q_2$ we define a new precoloring $\varphi'$ by coloring
		$e'_1$ and $e'_2$ by the color $c_1$, and also removing the color from
		any edge that is $\varphi$-colored $c_2$.
		
		By the induction hypothesis, there is an extension of $\varphi'$
		using colors $\{1,\dots, 2d+1\} \setminus \{c_2,c_3\}$. 
		From $\varphi'$, we obtain an edge coloring $f$ of $Q_1 \cup Q_2$
		by recoloring
		the edges of $Q_1$ and $Q_2$ that are $\varphi$-colored 
		$c_2$ by the color $c_2$. We color all edges
		between $Q_1$ and $Q_2$ by the color $c_3$, all the planes
		$Q_3, \dots, Q_{2k}$ correspondingly to how $Q_2$ is colored, and
		$Q_{2k+1}$ correspondingly to how $Q_1$ is colored.
		Now the edges between $Q_{2k+1}$ and $Q_{2k}$ can be colored with the color
		$c_3$, and every other edge of $D_1$ by some appropriate color missing at its endpoints.
		Thus by swapping colors on a bicolored $4$-cycle containing
		$e_1$ we obtain an extension of $\varphi$.
		
		Suppose now that there is no edge $e'_1 \in E(Q_1)$ adjacent to $e_1$,
		such that both $e'_1$
		and the corresponding edge $e'_2 \in E(Q_2)$ are uncolored and not adjacent
		to any edge precolored $c_1$. Then, since $Q_1 \cup Q_2$ is $(2d-2)$-regular
		and contains $2d-2$ precolored edges, $u_1$ is incident with $2d-2$
		edges $a_1,\dots, a_{2d-2}$ such that each $a_i$,
		or the corresponding edge of $Q_2$, is $\varphi$-colored by a color
		distinct from $c_1$, or uncolored and adjacent to an edge 
		$\varphi$-colored $c_1$. 
		In particular, if there is an edge of
		$Q_1 \cup Q_2$ colored $c_2$, then at most one edge 
		in each of $Q_1$ and $Q_2$
		is colored $c_2$.
		Moreover, since $G$ is triangle-free
		and $Q_1 \cup Q_2$ contains exactly $2d-2$ precolored edges, an edge 
		in $Q_1 \cup Q_2$ precolored $c_2$
		is not adjacent to an edge precolored $c_1$ in $Q_1 \cup Q_2$.
		
		If $Q_1$ contains an edge $a$ precolored $c_2$,
		then from the restriction of $\varphi$ to $Q_1 \cup Q_2$, we define a precoloring
		$\varphi'$ by recoloring $a$ and also the corresponding edge of $Q_2$ by
		the color $c_1$. Otherwise, if both $Q_1$ and $Q_2$ contain edges precolored $c_2$,
		then we define $\varphi'$ by recoloring these edges by the color $c_1$.
		Now, by the induction hypothesis, 
		there is an extension of the coloring $\varphi'$ using colors 
		$\{1,\dots, 2d+1\} \setminus \{c_2, c_3\}$. By recoloring 
		the edges that were recolored $c_1$ by the color $c_2$
		we obtain an extension of the restriction of $\varphi$ to $Q_1 \cup Q_2$.
		Next, we color $e_1$ by the color $c_1$ and all other
		edges between $Q_1$ and $Q_2$ by the color $c_3$.
		We color $Q_{2k+1}$ correspondingly to $Q_1$, and $Q_3,\dots, Q_{2k}$
		correspondingly to how $Q_2$ is colored, and then color the hitherto uncolored
		edges of $D_1$ as before to obtain an extension of $\varphi$.
		
		On the other hand, if no edge of $Q_1$ is colored $c_2$, then from the restriction of
		$\varphi$ to $Q_1 \cup Q_2$, we define a new precoloring $\varphi'$
		by coloring all edges adjacent to $e_1$ that are not precolored or adjacent
		to an edge colored $c_1$ in $G-E(D_1)$ by an arbitrary color from
		$\{1,\dots, 2d+1\} \setminus \{c_1,c_2,c_3\}$ so that the resulting precoloring
		is proper. 
		By the induction hypothesis, the obtained precoloring
		of $Q_1$ is extendable
		to a proper coloring using colors  $\{1,\dots, 2d+1\} \setminus \{c_2,c_3\}$,
		and the precoloring of $Q_2$ is extendable using colors  $\{1,\dots, 2d+1\} \setminus \{c_3, c_4\}$,
		where $c_4$ is some arbitrary color not appearing on an edge of $Q_2$.
		Note that no edge adjacent to $e_1$ is colored $c_1$ in this coloring.
		Hence, we can color $Q_{2k+1}$ correspondingly to how $Q_1$ is colored, all edges
		between $Q_1$ and $Q_{2k+1}$ by the color $c_2$, and all edges between $Q_1$ and $Q_2$
		by the color $c_3$ except that $e_1$ is colored $c_1$. Since not other planes
		in $G-E(D_1)$ contain precolored edges, it is now straightforward
		to obtain an extension of $\varphi$ from this partial coloring.

		\bigskip

		{\em (b) holds:}
		
		Suppose  that $Q_1$ and $Q_i$ contain all precolored edges of $G-E(D_1)$,
		where $i \notin \{1,2,3, 2k+1\}$.
		We take an extension of the restriction of $\varphi$ to $Q_1 \cup Q_i$ using
		colors $\{1,\dots, 2d+1\} \setminus \{c_3,c_4\}$, color $Q_2, Q_3, Q_{2k+1}$
		correspondingly to how $Q_1$ is colored, and all remaning
		planes in $G-E(D_1)$ by the colors $\{1,\dots, 2d+1\} \setminus \{c_3,c_4\}$
		so that the coloring agrees with $\varphi$.
		Next, we color the edges of $D_1$: the edges between $Q_2$ and
		$Q_3$ we color with the color in $\{1,\dots, 2d+1\} \setminus \{c_3,c_4\}$ missing
		at its endpoints, and all other edges of $D_1$ are colored $c_3$ and $c_4$ alternately,
		and starting with color $c_3$ at $Q_2$.
		This yields a coloring that agrees with $\varphi$ except for $e_1$ and $e_2$.
		We recolor these edges by $c_1$ and $c_2$, respectively, possibly by swapping on
		one or two bicolored $4$-cycles if necessary, to obtain an extension of $\varphi$.

		\bigskip

		{\em (c) holds:}
		
		The case when  $Q_2$, but not $Q_{2k+1}$ or $Q_1$, contains
		precolored edges can be dealt with as in the preceding
		pagragraph, so we omit the details here.
		
		\bigskip

	{\em (d) holds:}
		
	If $d=2$, then it is straightforward that
	$\varphi$ is extendable, because any partial $3$-edge coloring 
	of an odd cycle is extendable.
	If $d > 2$, then
	since $Q_2 \cup Q_{2k+1}$ contains exactly $2d-2$ precolored edges,
		it is straightforward that
		there are non-corresponding edges $e'_{2k+1} \in E(Q_{2k+1})$ and 
		$e'_2 \in E(Q_2)$ that are uncolored,
		adjacent to $e_1$ and $e_2$, respectively, and not adjacent to any edges
		of $Q_{2k+1} \cup Q_{2}$ precolored $c_2$ and $c_1$, respectively.
		
		From the restriction of $\varphi$ to $Q_{2k+1} \cup Q_1 \cup Q_2$, 
		we define a new precoloring $\varphi'$ of $Q_{2k+1} \cup Q_1 \cup Q_2$
		by in addition coloring $e'_{2k+1}$ by $c_2$, $e'_2$ by the color $c_1$,
		and the corresponding edges of $Q_1$ by colors $c_2$ and $c_1$
		respectively. By the induction hypothesis, there is an extension of $\varphi'$
		to $Q_{2k+1} \cup Q_1 \cup Q_2$ using colors 
		$\{1,\dots, 2d+1\} \setminus \{c_3,c_4\}$. From this coloring we may now obtain
		an extension of $\varphi$ by proceeding as before.


		\bigskip

		{\bf Subcase 2.2.3} {\em $e_1$ and $e_2$ are not incident with
		any common plane:}
		
		Suppose that $e_1$ is incident with $Q_1$ and $Q_2$, and $e_2$ is incident
		with $Q_j$ and $Q_{j+1}$.
		As in Subcase 2.1.2, we can distinguish between the following three cases:
		\begin{itemize}
		
		\item All precolored edges are contained in $Q_1$ and $Q_2$.
		
	\item One of $Q_1$ and $Q_2$, and one of $Q_j$ and $Q_{j+1}$, contain precolored edges.
		
		\item At most one of the planes $Q_1,Q_2, Q_j, Q_{j+1}$ contains precolored edges.
		
		\end{itemize}
		Moreover, in all these three subcases we may proceed precisely as in the corresponding subcases
		of Subcase 2.1.2. We omit the details.

				\bigskip
		
		{\bf Case 3.} {\em All precolored edges of $G-E(D_1)$ are contained in at least three planes:}

		In the case when no precolored edge of $D_1$ is incident with a plane
		containing precolored edges, then it is straightforward to obtain an extension
		by selecting uncolored edges in the planes that the precolored edges of $D_1$
		are incident with,
		so throughout we assume that this is not the case. Note further
		that since $G$ contains at least five precolored edges, $d \geq 3$.
		
		\bigskip
		
		{\bf Case 3.1.} {\em The precolored edges of $D_1$ have the same color
		under $\varphi$:}
		
		Suppose that $\varphi(e_1) = \varphi(e_2) = c_1$. We consider a number
		of different subcases.
		
		 \bigskip
		 
		 {\bf Subcase 3.1.1} {\em $e_1$ and $e_2$ are incident with the same two planes:}
		
		 Assume that $e_1$ and $e_2$ are incident with the same two planes
		$Q_1$ and $Q_2$.
		If the endpoints of $e_1$ and $e_2$ are adjacent via uncolored edges
		in both $Q_1$ and $Q_2$, then we define the precoloring $\varphi'$ from
		the restriction of $\varphi$ to $Q_1 \cup Q_2$
		by coloring these edges of $Q_1$ and $Q_2$ by the color $c_1$.
		We may now proceed as in Case 3 of Lemma \ref{cl:ODDoneempty} to obtain
		an extension of $\varphi'$, and thereafter swap colors on a
		bicolored $4$-cycle to obtain an extension of $\varphi$.
		
		Otherwise, if the endpoints of $e_1$ and $e_2$ are 
		not adjacent via uncolored edges
		in both $Q_1$ and $Q_2$. 
		Then, 
		since $d >2$, $Q_1 \cup Q_2$ contains at most $2d-3$ precolored edges
		and any two vertices in $G$ are contained in at most one $5$-cycle,
		it is not hard to see that
		there are independent edges $e'_1$ and $e'_2$ in $Q_1$,
		adjacent to $e_1$ and $e_2$, respectively,
		and such that neither
		these edges, nor the corresponding edges of $Q_2$
		are precolored or adjacent to edges precolored $c_1$ in $Q_1 \cup Q_2$.
		Hence, we may
		color these edges of $Q_1$ and $Q_2$ by the color $c_1$, and then proceed
		as in the preceding paragraph to obtain an extension of $\varphi$.

		\bigskip
		
		{\bf Subcase 3.1.2} {\em $e_1$ and $e_2$ are  incident with
		one common plane:}

		Suppose that $e_1$ and $e_2$ are incident with exactly one common 
		plane $Q_1$, and that $e_1$ is also incident with $Q_2$.
		If 
		there are at most $2d-4$ precolored edges in $Q_1 \cup Q_2$ 
		and at most $2d-4$ precolored edges in
		$Q_{2k+1} \cup Q_1$,
		then there are
		independent edges $e'_1$ and $e'_2$ in $Q_1$,
		adjacent to $e_1$ and $e_2$, respectively,
		and such that neither
		these edges, nor the corresponding edges of $Q_2$
		and $Q_{2k+1}$, respectively,
		are precolored or adjacent to edges precolored $c_1$ in $Q_1 \cup Q_2 \cup Q_{2k+1}$.
		Thus we may proceed as above to obtain an extension of $\varphi$.
		
	Suppose now instead that  $Q_1 \cup Q_{2k+1}$, say, contains exactly $2d-3$
		precolored edges. If there exist 
		independent edges $e'_1$ and $e'_2$ in $Q_1$,
		as described in the preceding paragraph, then we may proceed as in that case,
		so suppose that there are no two such edges.
		
		We first consider the case when we can choose exactly one such edge, that is,
		there is an edge $e'_1 \in E(Q_1)$ adjacent to $e_1$ but not $e_2$, 
		satisfying that $e'_1$ and the corresponding
		edge of $Q_2$ are not precolored or adjacent to edges of $Q_1 \cup Q_2$ that
		are precolored $c_1$. Moreover, there is no edge $e'_2$ adjacent to
		$e_2$ with analogous properties.
		Consider the precoloring $\varphi'$
		obtained from the restriction of $\varphi$ to $Q_{2k+1} \cup Q_1 \cup Q_2$
		by in addition coloring $e'_1$ and also the corresponding edge of $Q_2$
	by the color $c_1$. Now, since there is no uncolored edge $e'_2$ as 
	described above, it follows that
	all $d-2$ edges $a_1, \dots, a_{d-2}$ adjacent to $e_2$ in $Q_1$ satisfies
	that either $a_i$, or the corresponding edge of $Q_{2k+1}$, is $\varphi'$-precolored
	or adjacent to an edge colored $c_1$ under $\varphi'$.
	Moreover, since $Q_{2k+1} \cup Q_1$ contains at most $2d-2$ 
	$\varphi'$-precolored
	edges, every precolored edge of $Q_{2k+1} \cup Q_1$ satisfies this condition.
	Thus by properly coloring the uncolored edges adjacent to $e_2$, except the ones 
	that are adjacent to edges in $G-E(D_1)$ colored $c_1$,		by colors from 
	$\{1,\dots, 2d+1\} \setminus \{c_1, c_3,c_4\}$, we obtain a precoloring $\varphi''$
	from $\varphi'$. Note that any extension of the restriction of $\varphi''$ to
	$Q_{2k+1} \cup Q_1$ using colors $\{1,\dots, 2d+1\} \setminus \{c_3,c_4\}$ 
	does not use $c_1$ on an edge
	adjacent to $e_2$. Thus, since there is exactly one
	precolored edge in $G-E(D_1)$ that is not contained in $Q_{2k+1}\cup Q_1$,
	we may once again proceed as in Case 3 of 
	Lemma \ref{cl:ODDoneempty}
	to obtain a proper $(2d+1)$-edge coloring of $G$ that agrees with $\varphi''$,
	where no edge adjacent to $e_2$ is colored $c_1$, and where we color the edges 
	of $D_1$ so that all edges between any two given planes are colored by a 
	fixed color not appearing in these
	two planes.
		Thereafter we can swap colors on a
		bicolored $4$-cycle and recolor $e_2$ to obtain an extension of $\varphi$.
		
		Suppose now that neither an edge $e'_1$, nor an edge $e'_2$
		as described above exist in $Q_1$. Then the endpoints $u_1$ and $u_2$ of $e_1$ and $e_2$ in $Q_1$,
		respectively, are adjacent via an uncolored edge $e$ in $Q_1$, and the corresponding edges of
		$Q_{2k+1}$ and $Q_2$ are uncolored.
		Moreover, since both $Q_{2k+1} \cup Q_1$
		and $Q_1 \cup Q_2$ contains at most $2d-3$ precolored edges, it follows
		that there are $2d-4$ edges precolored $c_1$ in $Q_1$, the endpoints of which are
		adjacent to $u_1$ and $u_2$, respectively.
		Moreover, $Q_{2k+1}$ contains exactly one precolored edge, 
		and $Q_2$ contains exactly one precolored edge, so all precolored
		edges of $G-E(D_1)$ are contained in $Q_{2k+1} \cup Q_1 \cup Q_2$.
		
	Now, from the restriction of $\varphi$ to $Q_{2k+1}$
	we define a new precoloring $\varphi'$ by coloring all edges of $Q_{2k+1}$
	corresponding to edges of $Q_1$ colored $c_1$, by the color $c_1$,
	and thereafter color every edge adjacent to $e_2$ in $Q_{2k+1}$
	that is neither precolored
	nor adjacent to any edge precolored $c_1$ by an arbitrary color in 
	$\{1,\dots, 2d+1\} \setminus \{c_1, c_3,c_4\}$ so that the 
	resulting coloring is proper. 
	Next, we take an extension of $\varphi'$ using colors 
	$\{1,\dots, 2d+1\} \setminus \{c_3,c_4\}$. (Note that no edge adjacent
	to $e_2$ is colored $c_1$ in this extension.)
	Thereafter we color $Q_1$ correspondingly to how $Q_{2k+1}$ is colored 
	except that all edges colored $c_1$
	that are not precolored $c_1$ under $\varphi'$ are recolored $c_3$. 
	Denote the obtained partial coloring of $G$ by $f$.

	Now, if the precolored edge $b_2$ of $Q_2$ is colored $c_1$ under $\varphi$,
	then we color $Q_2$ correspondingly to how
	$Q_1$ is colored under $f$, except that the edge 
	$\varphi$-precolored $c_1$ is colored $c_1$.
	Therafter, we color $Q_3,\dots, Q_{2k}$ correspondingly to how 
	$Q_1$ is colored. We
	may now apply Lemma \ref{lem:oddcycle} to color the edges of $D_1$ 
	and thus obtain an extension of $\varphi$. (Since
	$e_1$ and $e_2$ are contained in different  cycles of $D_1$, 
	we can choose the coloring
	of $D_1$ so that it agrees with $\varphi$.)
	Otherwise, if $b_2$ is colored $c_5 \neq c_1$, then the corresponding 
	edge of $Q_1$
	is not colored $c_1$. We color $Q_2$ correspondingly to how $Q_1$ is colored
	except that
	we permute the colors $c_5$
	and the color of the corresponding edge of $Q_1$ under $f$. Again, we color 
	all the planes $Q_3,\dots, Q_{2k}$ correspondingly to how $Q_1$ is colored,
	and apply Lemma \ref{lem:oddcycle} to obtain an extension of $\varphi$.

		\bigskip
		
		{\bf Subcase 3.1.3} {\em $e_1$ and $e_2$ are not incident with
		any common plane:}
		
		Suppose that $e_1$ is incident with $Q_1$ and $Q_2$
		and $e_2$ is incident with $Q_j$ and $Q_{j+1}$, $j> 2$.
	Since each $Q_i$ is $(2d-2)$-regular and any pair of adjacent planes contain
	at most $2d-3$ precolored edges, it is straightforward that there are corresponding
		uncolored edges $e'_1 \in E(Q_1)$, $e'_2 \in E(Q_2)$ adjacent to $e_1$ but not to any
		other edge precolored $c_1$, and similarly for $e_2$.
		Hence, we may proceed as in Case 3 of Lemma \ref{cl:ODDoneempty} to obtain
		an extension of a precoloring $\varphi'$ of $G-E(D_1)$
		defined from $\varphi$ by coloring the selected edges adjacent to $e_1$ and $e_2$, respectively,
		by the color $c_1$ and removing the color $c_1$ from $e_1$ and $e_2$.
		From the extension of $\varphi'$,
		we obtain an extension of $\varphi$
		as before.
		
		\bigskip

		{\bf Case 3.2.} {\em The precolored edges of $D_1$ have different colors
		under $\varphi$:}
		
		Suppose that $\varphi(e_1) = c_1$ and $\varphi(e_2) = c_2$.
		
		\bigskip
		
		{\bf Subcase 3.2.1} {\em $e_1$ and $e_2$ are both incident with
		two common planes:}
		
		Assume that
		$e_1$ and $e_2$ are both incident with the same pair of planes $Q_1$
		and $Q_2$. Let $u_1$ and $u_2$ be the vertices of $Q_1$ that
		are incident with $e_1$ and $e_2$, respectively.
		
		If $Q_1 \cup Q_2$ contains at most $2d-4$ precolored edges, then
		there are  independent
		edges $e'_1 \in E(Q_1)$ and $e'_{2} \in E(Q_{1})$
		that are incident with $u_1$ and $u_2$, respectively, 
		such that neither $e'_1$ nor the corresponding edge $e''_1$ of $Q_2$ 
		is precolored or adjacent to an edge precolored $c_1$ in $Q_1 \cup Q_2$, and similarly
		for $e'_2$, the corresponding edge $e''_2$ of $Q_2$ and $c_2$.
		Thus, from the restriction
		of $\varphi$ to $Q_1 \cup Q_2$ we may define a new precoloring $\varphi'$
		by coloring these four edges by $c_1$ and $c_2$, respectively.
		We may now proceed as in Case 3 of Lemma \ref{cl:ODDoneempty} to obtain
		an extension of $\varphi'$, and thereafter we can obtain an extension of $\varphi$
		by swapping colors on two bicolored $4$-cycles.
		
		Suppose now that $Q_1 \cup Q_2$ contains $2d-3$ precolored edges, 
		so exactly one plane $D_i$, $i \neq 1,2$
		has exactly one precolored edge $a_i$; we assume $i \neq 3$.
		Then as in Claim \ref{cl:edge}, 
		there is either 
		\begin{itemize}
		
		\item[(i)] an edge
		$e'_1 \in E(Q_1)$ incident with $u_1$ but not $u_2$, such that
		$e'_1$ and the corresponding edge $e''_1$ of $Q_2$ are uncolored,
		and not adjacent to any edge  in $Q_1 \cup Q_2$ colored $c_1$, or
			
		\item[(ii)] an edge $e'_2 \in E(Q_1)$ incident with $u_2$ but not $u_1$,
			such that $e'_2$ and the corresponding edge $e''_2$ of $Q_2$ are uncolored,
			and not adjacent to any edge  in  $Q_1 \cup Q_2$ colored $c_2$.
		\end{itemize}
		
		Suppose e.g. that (ii) holds. Then we 
		define a new precoloring $\varphi'$ from the restriction of $\varphi$ to $Q_1 \cup Q_2$
		by coloring $e'_2$ and $e''_2$ by the color $c_2$. By removing
		the color $c_1$ from any edge that is colored $c_1$ under $\varphi'$, we obtain the
		precoloring $\varphi''$ of $Q_1 \cup Q_2$. Next, we
		take an extension of $\varphi''$ using colors $\{1,\dots, 2d+1\} \setminus \{c_1,c_3\}$,
		and then recolor all edges that are $\varphi$-colored $c_1$ by the color $c_1$
		to obtain the coloring $f$ which is an extension of $\varphi'$.
		We color all edges between $Q_1$ and $Q_2$
		by the color $c_3$ except $e_1$ which is colored $c_1$,
		color $Q_3$ correspondingly to how $Q_2$ is colored, and color
		the edges between $Q_2$ and $Q_3$
		by a color in $\{1,\dots, 2d+1\} \setminus \{c_1,c_3\}$
		missing at its endpoints.
		
		Next, consider the precolored edge $a_i$ of $Q_i$, and the corresponding edge
		$a_1$ of $Q_1$. If $f(a_1) =\varphi(a_i)$, then we color all the planes 
		$Q_4, \dots, Q_{2k+1}$
		correspondingly to how $Q_1$ is colored. Thereafter, we 
		color the edges between $Q_3$ and $Q_4$ similarly to how the edges between $Q_1$ and $Q_2$ 
		are colored, and then
		color the remaining uncolored paths in $D_1$ using two colors not appearing
		at the endpoints of these paths.
		Finally, we swap colors on a bicolored $4$-cycle containing $e_2$ to obtain an extension
		of $\varphi$.

		Otherwise, if $f(a_1) \neq \varphi(a_i)$, then 
		we color the planes $Q_4, \dots, Q_{2k+1}$ correspondingly to how $Q_1$ is colored, 
		except that we permute the colors $f(a_1)$ and
		$\varphi(a_i)$. Then we 
		color the edges between $Q_3$ and $Q_4$ with the color $c_3$,
		and consider the subgraph $H$ consisting of  the
		edges of $D_1$ with endpoints in two consecutive planes in the
		sequence $Q_4, \dots, Q_{2k+1}, Q_1$.
		If we define a list assignment for these edges by for every edge including
		the colors from $\{1,\dots, 2d+1\}$ that do not appear on any adjacent edges, then
		each edge, except the ones with endpoints in $Q_1$ and $Q_{2k+1}$, gets a list of size
		at least two. Hence, $H$ is list edge colorable from these lists.
		This yields an edge coloring of $G$ that agrees with $\varphi$ except for $e_2$.
		Finally, we swap colors on a bicolored $4$-cycle containing $e_2$ to obtain an extension
		of $\varphi$.

		\bigskip
		
		{\bf Subcase 3.2.2} {\em $e_1$ and $e_2$ are incident with
		exactly one common plane:}
		
		Suppose now instead that $e_1$ and $e_2$ are incident to 
		exactly one common plane, say $Q_1$,
		where $e_1$ in addition also is incident with $Q_2$.
		If there are uncolored corresponding
		edges $e'_1 \in E(Q_1)$ and $e'_{2k+1} \in E(Q_{2k+1})$ 
		that are incident with $e_2$ but not to any other edge
		precolored $c_2$, and similar edges for $e_1$ and the color $c_1$
		in $Q_1$ and $Q_2$, respectively, which are disjoint from $e'_1$,
		then we proceed as above:
		we can obtain an extension by coloring the edges adjacent to $e_1$ and $e_2$ by colors
		$c_1$ and $c_2$, respectively, and
		then proceed as in  Case 3 of Lemma \ref{cl:ODDoneempty},
		as before.
		
		Now, any two adjacent planes
		contain at most $2d-3$ precolored edges,
		so if there are no edges as described in the preceding paragraph, then
		all precolored edges are contained in $Q_{2k+1} \cup Q_1 \cup Q_2$,
		and $e_1$ and $e_2$ are adjacent to a common vertex $u_1 \in V(Q_1)$.
		Moreover, $u_1$ is incident with $2d-4$ edges colored by distinct colors from
$\{1,\dots, 2d+1\} \setminus \{c_1,c_2,c_3,c_4\}$, 
$Q_{2k+1}$ contains exactly one
precolored edge, and $Q_2$ contains exactly one precolored edge. 
Moreover, these precolored edges in $Q_{2k+1} \cup Q_2$ are 
either adjacent to vertices corresponding
		to $u_1$, or colored $c_2$ and $c_1$ respectively, and adjacent to edges that
		are incident with $u_1$.
		We consider some different cases, depending on the colors of the precolored
		edges of $Q_1$ and $Q_2$.
		
		Suppose first that the precolored edge of $Q_{2k+1}$ is colored $c_2$,
		and that $Q_2$ contains an edge precolored $c_1$. 
		We color all edges of $Q_{2k+1}$ adjacent to $e_2$ that are not precolored
		or adjacent to an edge precolored $c_2$ by arbitrary colors from 
		$\{1,\dots, 2d+1\} \setminus \{c_1, c_2,c_3\}$ so that the resulting precoloring is proper, 
		and similarly for $Q_2$ but with $c_1$
		in place of $c_2$.
		Next, we take an extension of the resulting precoloring $\varphi'$
		of $Q_{2k+1} \cup Q_1 \cup Q_2$, where we use colors 
		$\{1,\dots, 2d+1\} \setminus \{c_1,c_3\}$ for $Q_{2k+1}$, 
		$\{1,\dots, 2d+1\} \setminus \{c_1,c_2\}$ for $Q_1$,
		and $\{1,\dots, 2d+1\} \setminus \{c_2,c_3\}$ for $Q_2$.
		We then color the edges between $Q_{2k+1}$ and $Q_1$ by $c_1$ except $e_2$ which is colored $c_2$,
		the edges between $Q_{1}$ and $Q_2$ by $c_2$ except $e_1$ which is colored $c_1$.
		Next, we color the planes $Q_3,\dots, Q_{2k}$ correspondingly using colors
		$\{1,\dots, 2d+1\} \setminus \{c_3,c_4\}$, and  all remaining uncolored edges by $c_3, c_4$ alternately.
		This yields an extension of $\varphi$.
		
		Now, if one of the colors $c_1$ and $c_2$ does not appear in $G-E(D_1)$, say $c_2$, 
		then from the restriction of $\varphi$ to $Q_1 \cup Q_2$, we 
		define a new precoloring $\varphi'$ by properly coloring all the edges adjacent to $e_1$ that
		are not precolored or adjacent to an edge colored $c_1$ by arbitrary colors
		in $\{1,\dots, 2d+1\} \setminus \{c_1,c_2,c_3\}$ so that the resulting coloring is proper.
		We take an extension of $\varphi'$ using colors 
		$\{1,\dots, 2d+1\} \setminus \{c_2,c_3\}$, and
		an extension of the restriction of $\varphi$ 
		to $Q_{2k+1}$
		using colors $\{1,\dots, 2d+1\} \setminus \{c_2,c_3\}$. 
		Thereafter, we color all edges between $Q_{2k+1}$ and $Q_1$ by the color $c_2$,
		the edges between $Q_1$ and $Q_2$ by the color $c_3$ except that $e_1$ is colored $c_1$.
		Since no other planes in $G-E(D_1)$ contain precolored edges, it is now straightforward
		to construct an extension of $\varphi$ from the obtained partial edge coloring of $G$.

		Finally, if the precolored edge of $Q_{2k+1}$ is colored $c_1$, and the edge of $Q_2$ is
		colored $c_2$, then we proceed similarly, but simply take extensions of the restriction of
		$\varphi$ to $Q_{2k+1}$ using colors $\{1,\dots, 2d+1\} \setminus \{c_2,c_3\}$,
		of the restriction of $\varphi$ to $Q_1$ using colors 
		$\{1,\dots, 2d+1\} \setminus \{c_1,c_2\}$
		and of the restriction of
		$\varphi$ to $Q_{2}$ using colors $\{1,\dots, 2d+1\} \setminus \{c_1,c_3\}$.

		\bigskip
		
		{\bf Subcase 3.2.3} {\em $e_1$ and $e_2$ are not incident with
		any common plane:}
		
		It remains to consider the case when  $e_1$ and $e_2$ are not incident to
		a common plane. Here we may proceed precisely as in Subcase 3.1.3, so once
		again we omit the details.
		This concludes the proof of this lemma.
		\end{proof}
		

		\bigskip


\section{Extending a precoloring of a distance-$4$ matching in $C^d_{2k}$}

In this last section we consider the problem of extending a precoloring of 
$C^d_{2k}$ 
where the precolored edges form a matching.

\begin{theorem}
If $\varphi$ is a $2d$-edge coloring of a distance-$4$ matching of $G = C^d_{2k}$, 
then $\varphi$ can be extended to a proper $2d$-edge coloring of $G$.
\label{D3M}
\end{theorem}
\begin{proof}
Let $\varphi$ be a $2d$-edge precoloring of a distance-$4$ matching $M$ of $G$,
and let $D_1, \dots D_k$ be the dimensions of $G$.
We define the edge coloring $f$ of $G$ by properly coloring all edges of
$D_j$ by $2j-1$ and $2j$, so that all corresponding edges have the same color.
The resulting coloring satisfies that
every $4$-cycle in $G$ is bicolored since corresponding edges have the same color. 

We shall describe a procedure for obtaining a required coloring $f'$  that agrees
with $\varphi$. 
For all precolored edges we shall use transformations on some bicolored $4$-cycles.
As we shall see, if $e, e' \in M$, then the cycles used for transformations 
involving $e$ will be edge-disjoint from cycles used for $e'$.

Consider an arbitrary precolored edge $e \in M$. We consider some different cases.
\begin{itemize}

\item[(i)] If $f(e) = \varphi(e)$, then we are done;

\item[(ii)] If $f(e) \neq \varphi(e)$, and there is a bicolored 
$4$-cycle containing $e$, and where color $\varphi(e)$ appears, then we interchange colors on this bicolored $4$-cycle;

\item[(iii)] If none of the two previous conditions hold, then there are 
two edges $e_1$ and $e_2$, both of which are adjacent to $e$, 
and contained in the same dimension as $e$, such that $\varphi(e)=f(e_1) = f(e_2)$. 
By interchanging colors on two disjoint $4$-cycles, containing $e_1$ and $e_2$ respectively, we obtain a coloring $f_1$, where $e$ is contained in a 
bicolored $4$-cycle with the color $\varphi(e)$. 
Thus by interchanging colors on this $4$-cycle,
we obtain a coloring $f_2$ satisfying that $f_2(e)=\varphi(e)$.
\end{itemize}

Note that all edges used in the transformations (i) - (iii) are at distance 
at most $1$ from $e$. Thus if $e$ and $e'$ are distinct edges of $M$, and we perform
one of the transformations (i)-(iii) for both edges, then the edges involved in the
transformations concerning $e$ will be edge disjoint 
from the ones used for $e'$, since the precolored edges form a distance-$4$ matching.

Hence, we can repeat the above process for any precolored edge of $G$ to 
obtain the required coloring $f'$.
\end{proof}

We believe that Proposition~\ref{D3M} might be true if we precolor a distance-$3$ instead of a distance-$4$ matching, but if $e$ and $e'$ are distinct edges of $M$, then the edges involved in the transformations for $e$ may not necessarily 
be disjoint from the one used for $e'$, and thus we cannot apply our technique here;  we state the following conjecture.
\begin{conjecture}
If $\varphi$ is an edge precoloring of a distance-$3$ matching of $C^d_{2k}$, then $\varphi$ can be extended to a proper 
$4$-edge coloring of $C^d_{2k}$.
\end{conjecture}

Note that Proposition~\ref{D3M} becomes false if we precolor a distance-$2$ matching; for instance, consider a vertex $v$ of degree $2d$ such that every edge incident with $v$ is uncolored but there is a fixed color $c \in \{ 1, \dots, 2d\}$ satisfying that every edge
incident with $v$ is adjacent to another edge colored $c$. 
If $f$ is an extension of $\varphi$, then since $v$ has degree $2d$, exactly one edge incident 
with $v$ is colored $c$, but such a coloring cannot be proper. \\

\section*{Acknowledgements}
Petros thank the International Science Program in Uppsala, Sweden, for financial support.
Casselgren was supported by a grant from the Swedish Research council VR
(2017-05077).

\end{document}